\documentclass{article}
\usepackage{amsmath}
\usepackage{amsthm}
\usepackage{amsfonts}
\usepackage{yfonts}
\usepackage{soul}

\usepackage{anyfontsize}
\usepackage{tikz}
\usepackage{amssymb}
\usepackage{bbm} 
\usepackage{color}
\usepackage{xcolor}
\usepackage{hyperref}
\usepackage{multirow}
\usepackage{graphicx}
\usepackage{comment}

\newcommand{\bb}[1]{{\mathbb {#1}}}
\newcommand{\bbm}[1]{{\mathbbm {#1}}}

\newcommand{\mc}[1]{{\mathcal {#1}}}

\newcommand\abs[1]{\left\vert #1\right\vert}
\newcommand\crt[1]{\left[ #1 \right]}
\newcommand\prt[1]{\left( #1 \right)} 
\newcommand{\tod}{\overset{(d)}{\longrightarrow}}

\newcommand{\todtwo}{\overset{L^2}{\longrightarrow}}
\newcommand{\Var}{\ensuremath{\bb{V}\mathrm{ar}}}
\newcommand{\var}{\ensuremath{\bb{V}\mathrm{ar}}}
\newcommand\Ind[1]{ \mathbbm{1}_{#1}} 
\newcommand\ind[1]{ \mathbbm{1}_{#1}} 

 \newcommand\norm[1]{\left\Vert #1\right\Vert}
\newcommand\chv[1]{\{\,#1\,\}} 

\newtheorem{theorem}{Theorem}[section]         
\newtheorem{lemma}[theorem]{Lemma}
\newtheorem{corollary}[theorem]{Corollary}
\theoremstyle{definition}
\newtheorem{definition}[theorem]{Definition}

\theoremstyle{remark}
\newtheorem{remark}[theorem]{Remark}

\makeatletter
\def\namedlabel#1#2{\begingroup
    #2%
    \def\@currentlabel{#2}%
    \phantomsection\label{#1}\endgroup
}

\makeatother

\title{A cohesive account on the ergodic behaviours and scaling limits of Random Walks in Cooling Random Environments}
\author{Luca Avena \footnote{Università degli Studi di Firenze, Viale Morgagni, 67/A, 50134, Firenze, 055 2751400, IT. Email:~\href{mailto:luca.avena@unifi.it}{\texttt{luca.avena@unifi.it}}.} \and Conrado da Costa\footnote{\raggedright Instituto de matemática estatística e ciência da computação, Universidade de São Paulo, Rua do Matão, Cidade Universitária, Butantã,  1010, São Paulo, 05508-090, BR. Email:~\href{mailto:conrado.da-costa@durham.ac.uk}{\texttt{conrado.dacosta@ime.usp.br}}.}}

\begin{document}
\maketitle
\begin{abstract}
Transport in disordered media is a central theme in probability and statistical physics, where randomness in the underlying medium produces phenomena such as localization, anomalous scaling, and slow relaxation. A paradigmatic model for transport in disordered media is that of Random Walks in Random Environments (RWRE), which has been extensively studied since the 1970's and is by now well understood in one dimension.

More recently, several works have explored perturbations of models of transport in disordered media aimed at interpolating between static disorder and fully homogenized dynamics. 
Random walks in cooling random environments (RWCRE), introduced in this context, constitute a key example: the environment is dynamically resampled at prescribed times and kept fixed in between, giving rise to a delicate ``quasi-ergodic'' structure in time allowing to interpolate between homogeneous random walks and classical RWRE.

The purpose of this paper is twofold. A first goal is to offer an original survey on the main results for RWCRE in 1d: recurrence criteria, law of large numbers, large deviations, and fluctuation phenomena across different resampling regimes. Those results have been derived in a series of recent works, \cite{AdH17springer, AveChidCosdHol19, Yon19, AveChidCosdHol22,AvedCosPet23, AvedCos24game, dCosPetXie23b} and are complemented here with a number of new statements aiming at presenting a unified phenomenological picture. As a second goal, we try to extract a coherent conceptual picture that highlights structural mechanisms -- such as ergodic limits, persistence under perturbations and replacement principles -- that extend beyond the specific setting of RWCRE and are relevant for a broader class of disordered systems that can be perturbed by introducing independent resetting in the same fashion.
\end{abstract}


(Secondary)

\tableofcontents

\section{RWRE, cooling, and perturbation via independent resetting}

Random walks in random environments (RWRE) constitute one of the central models
for transport in disordered media. In its simplest form, a RWRE describes the
motion of a particle on $\bb{Z}^d$ whose transition probabilities are sampled
once from a random environment and then kept fixed for all times. This model exhibits a rich phenomenology, including trapping,
sub-ballistic motion, and anomalous fluctuations. Owing to several
decades of intensive research, the qualitative and quantitative behaviour of RWRE is by now well understood in dimension one, while various challenging open problems remain in other settings, see e.g. \cite{Sol75, Zei04,MetNehZoy26,GroRamSagZan26} 

A natural question, motivated both by physics and probability, is how robust
these phenomena are under dynamic perturbations of the environment. One particularly
transparent class of perturbations consists in \emph{independent resetting} or
\emph{resampling} of the environment at prescribed times. In this setting, the
environment is refreshed intermittently, independently of the past, while the
walk evolves according to the currently active environment between two resampling events. Such perturbations weaken the effects of the space-time correlations without fully
destroying disorder and introduce a tunable time scale controlling the strength
of homogenization.

Random walks in cooling random environments (RWCRE), introduced 
in \cite{AdH17springer}, form a prominent example of this mechanism. In RWCRE, the
environment is resampled along an arbitrary sequence (possibly random) of times, called the
\emph{cooling sequence}, and is kept fixed in between. This construction
interpolates between two well-understood extremes: homogeneous random walks,
corresponding to resampling at every step, and classical RWRE, corresponding to
no resampling at all. Depending on the growth of the resampling intervals, RWCRE
exhibits a variety of cross-over regimes leading to new behaviour or persistence of the features of one of the two interpolated media (static and homogeneous).
Let us stress that such a cooling random environment is a specific tractable type of dynamic random media, and in recent years much attention has been devoted in trying to analyse RW models in evolving media on lattices, for progresses in this growing literature we refer the reader to \cite{avena_thesis_10, DENHOLLANDER2014785, MR3108811, MR2786643, HDHS15, BHT18, OtaMil18, HilKioTei20, Avena4, Avena3, Avena2}
In this manuscript we provide a unified overview of the results obtained so far for the specific RWCRE model, and at the same time we use RWCRE as a guiding example to identify structural mechanisms that extend beyond this specific model. In
particular, we seek to disentangle which asymptotic phenomena are genuinely tied to the details of RWCRE and which arise from more general features of resampling dynamics. To this end, we introduce a general framework that captures the essential
structure of processes evolving under resampling, while allowing the underlying
dynamics and the resampling mechanism to be varied independently. RWCRE will
later be recovered as a particular instance of this framework. Before entering  the details we give next an outline of the paper.

\subsection{Outline of the paper} \label{ssec:outline}
Apart from the introductory Section \ref{sec:frame}, the remaining ones address structurally different mathematical questions. We have hence tried to present them in a self-contained manner, so that the main messages and tools can be appreciated separately.

\begin{itemize}
    \item Section \ref{sec:frame} ({\bf Models \& underlying subtle mechanisms }) starts with precise definitions of the general framework, Sec. \ref{Genframe}, and the two specific 1d RW models in random environments and how they relate, Sec. \ref{s:rwre_rwcre_kernel}. We then discuss the structure of the possible resampling maps, Sec. \ref{maps}, the type of behaviour to be expected for the main asymptotics of the RWCRE, \ref{s:phenomenology}, and the methods of proofs of general interest behind these results    \ref{s:methods}. 
    \item Section \ref{sec:homogenization} ({\bf Homogenization for non-cooling maps}) discusses how to show that homogenization takes place for the easiest class of resampling maps where no effective cooling is in place. We will in particular present a
    brief original perspective on this simple class of resettings following a large deviation perspective. 
    \item Section \ref{sec:recurrencevstransience} ({\bf Recurrence in 1d}) discusses the current understanding of recurrence and transience properties of RWCRE in 1d in relation to the classical Solomon's criterion known for RWRE.

    \item Section \ref{sec:ergodic} ({\bf Cooling ergodic theory \& game of mass}) discusses, under both the quenched and annealed laws, the Law of large Numbers (LLN) and Large Deviation Principle (LDP) for the average displacement. 
    These problems are in fact related to the ``almost ergodic'' sums emerging from the RWCRE model. For these general sums we discuss weak and strong convergence statements for centred variables, Section~\ref{sec:gradual-sums}, as well as regularity tools to determine explicit limits in the non-centred case. The latter is what we refer to as \emph{the game of mass}, which we present in Section \ref{ssec:game_of_mass} specialized to determining the value of the asymptotic speed.
    As far as the RWCRE in 1d is concerned, apart from reviewing the results mainly derived in \cite{AdH17springer, AveChidCosdHol19, AvedCos24game}, we prove in Section \ref{WLLNannLDP} a general rate $n$ annealed LDP, Theorem \ref{annLDPthm}, which was not discussed in previous works.    
    \item Section \ref{sec:Fluctuations} ({\bf Scaling limits}), discusses results and underlying methods of proofs as far as fluctuations  are concerned. We present the various results across different regimes developed in \cite{AdH17springer,AveChidCosdHol22, AvedCosPet23, dCosPetXie23b,Yon19} in a unified approach,  from the cohesive perspective of the proof methodology.  In Section \ref{Ssec:Replacement_works} we present in a novel unifying theorem Theorem~\ref{thm:rlpts} the caracterization of the scaling limits for the regimes where a certain replacement method holds. In Section~\ref{Ssec:replacement_fails},  we treat the more difficult and subtle regimes where the replacement method fails. 
\end{itemize}

\section{Models, stucture, questions \& methods}\label{sec:frame}

\subsection{General framework}
\label{Genframe}
We begin by describing an abstract setting that underlies RWCRE and related
models.
Let $\bb{X} = (\bb{X},\mathcal{B})$ be a separable topological abelian group
endowed with its Borel $\sigma$-algebra. In particular, the addition map
$+\colon \bb{X}\times\bb{X}\to\bb{X}$ is measurable and commutative.
This setting is both natural and sufficiently general: it includes the relevant state space for RWCRE, $\bb{Z}$ or more generally $\bb{R}^d$, and ensures
that increments of the underlying processes can be combined additively in an
order-independent and measurable manner. Moreover, several of the limit theorems
that motivate this framework—laws of large numbers, central limit theorems, and
concentration inequalities—are inherently formulated in additive, commutative
settings.

Let $(\Omega,\mathcal{F},\bb{P})$ be a probability space. For each
$j\in\bb{N}$, let $
\mathbf{X}^{(j)} = (X^{(j)}_t,\ t\ge 0)$ 
be an $\bb{X}$-valued stochastic process, meaning that for every $t\ge 0$ the
map $X^{(j)}_t\colon \Omega\to\bb{X}$ is measurable. The family of underlying processes will be denoted by
\begin{equation}\label{underlying}
    \mathfrak{X}=(\mathbf{X}^{(j)},\ j\in\bb{N}).
\end{equation}
We further fix $\tau = (\tau(n),\ n\in\bb{Z}_+ \cup \{\infty\})$, a sequence satisfying
\begin{equation}\label{basic-resampling-tau}
\tau(0)=0, \qquad \tau(n)\leq \tau(n+1), 
\qquad\text{and}\qquad
\lim_n \tau(n) = \infty.
\end{equation}
A function $\tau:\bb{Z}_+ \to \bb{Z}_+ \cup \{\infty\}$ satisfying \eqref{basic-resampling-tau} will be called a \emph{resampling map}.
Combining $\mathfrak{X}$ with
the resampling map $\tau$, we define the associated \emph{resampled process}
$X=(X_t,\ t\ge0)$ by
\begin{equation}\label{cooling-general-tau}
X_t = X_t(\mathfrak{X},\tau)
:= \sum_{j\geq 1} X^{(j)}_{(\tau(j)\wedge t)-(\tau(j-1)\wedge t)}.
\end{equation}
From~\eqref{basic-resampling-tau} it follows that only finitely many indices $j$ satisfy
$\tau(j)\le t$, so the sum in~\eqref{cooling-general-tau} is finite and hence
well defined.

Among the resampling maps $\tau$, a special subclass is formed by those for which the block
lengths
\begin{equation}\label{effective-cooling}  
T_j := \tau(j)-\tau(j-1)
\quad\text{
satisfy } \quad T_j\to\infty \quad\text{ as } \quad j\to\infty, 
\end{equation}
we call such a $\tau$ a \emph{cooling map}. This
terminology reflects the intuition that, in the absence of resampling, the system
evolves in a frozen environment. Large values of $T_j$ correspond to long stretches
during which the dynamics remain unchanged, and as $T_j$ diverges the evolution
approaches a genuinely frozen regime.

The general objective of this framework is to analyse the asymptotic behaviour of
the process $X=(X_t(\mathfrak{X},\tau),\ t\ge0)$ under different choices of the
underlying processes $\mathfrak{X}$ and the resampling map $\tau$.

The ``patchwork'' representation in  \eqref{cooling-general-tau} reduces much of the analysis to understanding 
how the asymptotic behaviour of $X$ emerges from the combined effect of its constituent pieces.

\subsection{RWRE and RWCRE}
\label{s:rwre_rwcre_kernel}

We specialise the abstract patchwork construction
\eqref{cooling-general-tau} to the one-dimensional nearest-neighbour setting
that underlies RWRE and RWCRE.

\subsubsection{Static environments}

An \emph{environment} is a field
$\omega=(\omega(x):x\in\bb{Z})\in(0,1)^{\bb{Z}}$, where $\omega(x)$ represents
the probability to jump from $x$ to $x+1$ (and $1-\omega(x)$ is the
probability to jump from $x$ to $x-1$). To ease the exposition\footnote{Most of the discussed results can be properly generalized to e.g. spatially ergodic settings.} we will work throughout the paper under the assumption that $\omega$ is an i.i.d. sequence with law
\begin{equation}\label{eq:mu_iid}
\mu=\alpha^{\bb{Z}},
\end{equation}
where $\alpha$ is a non-degenerate probability measure on $(0,1)$;  and
that uniform ellipticity is in place, i.e.,
\begin{equation}\label{eq:ue}
\exists\,\mathfrak{c}>0:\qquad \alpha\big(\mathfrak{c}\le \omega(0)\le 1-\mathfrak{c}\big)=1.
\end{equation}

\begin{definition}[RWRE: quenched and annealed laws]\label{def:rwre}
Let $\omega$ be sampled from $\mu$. The \emph{random walk in random environment}
(RWRE) is the Markov chain $Z=(Z_n)_{n\in\bb{N}_0}$ on $\bb{Z}$ with transition
kernel
\begin{equation}\label{eq:rwre_kernel}
P^{\omega}(Z_{n+1}=x+e\mid Z_n=x)=
\begin{cases}
\omega(x), & e=+1,\\
1-\omega(x), & e=-1,
\end{cases}
\qquad x\in\bb{Z},\ n\in\bb{N}_0.
\end{equation}
We write $P_x^{\omega}$ for the \emph{quenched} law of $Z$ started at $x$ and
\begin{equation}\label{eq:annealed_rwre}
P_x^{\mu}(\,\cdot\,):=\int_{(0,1)^{\bb{Z}}} P_x^{\omega}(\,\cdot\,)\,\mu(d\omega)
\end{equation}
for the corresponding \emph{annealed} law (with expectations $E_x^{\omega}$ and
$E_x^{\mu}$).
\end{definition}

\subsubsection{Cooling dynamic environments}\label{ssec:cre}

Let $\tau:\bb{N}_0\to\bb{N}_0$ be a \emph{resampling map} as defined in \eqref{basic-resampling-tau}.
Let $\bar{\omega}=(\omega_k, k\in\bb{N})$ be an i.i.d. sequence of environments sampled from 
the law $\mu^{\bb{N}}$. The pair $(\bar{\omega},\tau)$ defines a space--time dynamic  environment by assigning $\omega_k$ to the time interval
$[\tau(k-1),\tau(k))$. Equivalently, at time $n$ the walk sees the environment
with index
\begin{equation}\label{eq:ell_def}
\ell_n:=\inf\{k\in\bb{N}:\tau(k)>n\}.
\end{equation}

\begin{definition}[RWCRE: kernel formulation]\label{def:rwcre_kernel}
Given a resampling map $\tau$ and an i.i.d. environment sequence
$\bar{\omega}\sim\mu^{\bb{N}}$, the \emph{random walk in cooling random
environment} (RWCRE) is the Markov chain $X=(X_n)_{n\in\bb{N}_0}$ on $\bb{Z}$
with transition kernel
\begin{equation}\label{eq:rwcre_kernel}
P^{\bar{\omega},\tau}(X_{n+1}=x+e\mid X_n=x)=
\begin{cases}
\omega_{\ell_n}(x), & e=+1,\\
1-\omega_{\ell_n}(x), & e=-1,
\end{cases}
\quad x\in\bb{Z},\ n\in\bb{N}_0,
\end{equation}
where $\ell_n$ is given by \eqref{eq:ell_def}. We denote by
$P_x^{\bar{\omega},\tau}$ the quenched law and by
\begin{equation}\label{eq:annealed_rwcre}
P_x^{\mu,\tau}(\,\cdot\,):=\int_{[(0,1)^{\bb{Z}}]^{\bb{N}}}
P_x^{\bar{\omega},\tau}(\,\cdot\,)\,\mu^{\bb{N}}(d\bar{\omega})
\end{equation}
the corresponding annealed law.
\end{definition}

A key feature of RWCRE is that its position at time $n$ can be decomposed into
independent contributions from successive blocks. Define the \emph{refreshed
increments} and the \emph{boundary increment} by
\begin{equation}\label{eq:space_incr}
Y_k:=X_{\tau(k)}-X_{\tau(k-1)},\quad k\in\bb{N},
\qquad
\bar{Y}^{\,n}:=X_n-X_{\tau(\ell_n-1)} ,
\end{equation}
and the \emph{boundary running time} by
\begin{equation}\label{eq:time_incr}
\bar{T}^{\,n}:=n-\tau(\ell_n-1).
\end{equation}
Then
\begin{equation}\label{eq:time_partition}
\sum_{k=1}^{\ell_n-1} T_k+\bar{T}^{\,n}=n,
\end{equation}
and, by construction,
\begin{equation}\label{eq:X_dec}
X_n=\sum_{k=1}^{\ell_n-1} Y_k+\bar{Y}^{\,n},\qquad n\in\bb{N}_0.
\end{equation}

The vector $(Y_1,\dots,Y_{\ell_n-1},\bar{Y}^{\,n})$ has independent components,
each component being distributed as a RWRE increment over the corresponding time
length in the corresponding environment. More precisely, for any measurable
$f:\bb{Z}\to\bb{R}$,
\begin{equation}\label{eq:increment_identification}
E_0^{\bar{\omega},\tau}\big[f(Y_k)\big]=E_0^{\omega_k}\big[f(Z_{T_k})\big],
\qquad
E_0^{\bar{\omega},\tau}\big[f(\bar{Y}^{\,n})\big]=E_0^{\omega_{\ell_n}}\big[f(Z_{\bar{T}^{\,n}})\big].
\end{equation}

\begin{figure}[htbp]
\begin{center}
\includegraphics[bb = 0 0 595 841, trim=2.8cm 11cm 2.8cm 14cm, scale=0.65]{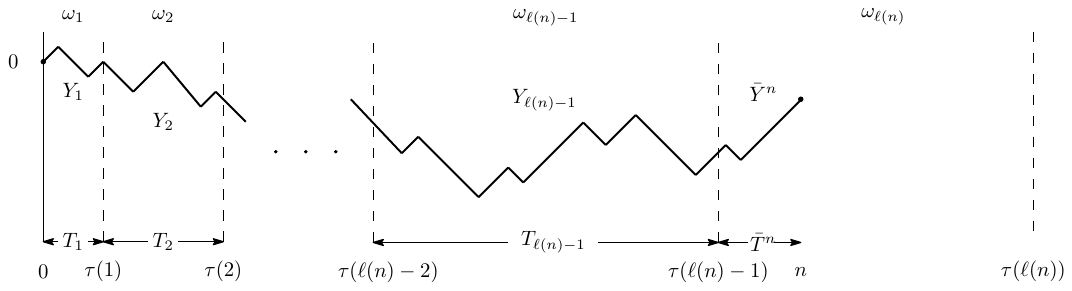}
\vspace{4mm}
\hspace{10mm}\caption{The decomposition of RWCRE in pieces of RWRE as presented in~\eqref{eq:X_dec}.}
\label{fig:RWCRE}
\end{center}
\end{figure}

\paragraph{Relation to the patchwork construction}
As shown in \eqref{eq:X_dec}, \eqref{eq:increment_identification} and in Figure \eqref{fig:RWCRE}, the kernel definition of RWCRE given in \eqref{eq:rwcre_kernel} fits the abstract resampling process
\eqref{cooling-general-tau} by taking the underlying family
$\mathfrak{X}=(\mathbf{X}^{(k)},k\geq 1)$ to be independent copies of the RWRE
increment process in environments $\omega_k$, and concatenating the increments
over the block lengths $(T_k)_{k\ge1}$ determined by $\tau$. In fact, both the quenched and annealed constructions starting from $x \in \bb{Z}$ fit into this framework provided we follow the inductive definition that sets $X_{\tau(0)} := x$ and for each $j \in \bb{N}$ the 
variable $X^{(j)}$ used to build $\mathfrak{X}$ is defined as the displacement of the quenched/annealed random walk starting from $X_{\tau(j-1)}$.
For convenience we shall consider that all random variables we consider --- be it the abstract family of processes $\mathfrak{X} = (X^{(j)}_n, n \in \bb{Z}_+, j \in \bb{N})$, the underlying environment $\omega = (\omega(x):x\in\bb{Z})$, the resampling maps $\tau$, or their increments $(T_k, k \in \bb{N})$ --- are defined in the same probability space $(\Omega, \mc{F}, \bb{P})$. Given random variable $W$ We shall denote by $ \bb{E}[W]$ the expectation of $W$ with respect to $\bb{P}$.

The patchwork representation \eqref{cooling-general-tau} shows that RWCRE is governed by
two ingredients: the underlying-block laws in $\mathfrak{X}$ (i.e.\ the RWRE
increments that make up each block $X^{(j)}, j \in \bb{N}$) and the resampling map $\tau$ (or equivalently the increment sequence $(T_k, k\ge1)$).

The main goal is to identify, for different choices of environment law, what sort of limiting behaviours emerge for the
process defined by \eqref{cooling-general-tau}  when choosing structurally different resampling maps $\tau$.

\subsection{Types of resampling maps}
\label{maps}
Heuristically, one may capture a cross-over phenomenon from homogeneous random walk  to static RWRE by comparing 
$X(\mathfrak{X},\tau_1)$ with $X(\mathfrak{X},\tau_\infty)$, where 
for every $n \in \bb{N}$,
$\tau_1(n):= n$ and  $\tau_\infty(n) := \infty$.

At first, using the notion of $f(n) \sim g(n)$ if $\lim_n f(n)/g(n) = 1$, one focuses on simple maps that interpolate monotonically between $\tau_1$ and $\tau_\infty$ such as the polynomial maps with  $\tau(n) \sim A n^a$ or the exponential maps with  $\tau \sim B e^b$ for different choices of $A,a>0$  or $B,b>0$.
Yet, depending on the type of questions under examination, much more general families of maps may be considered and different suitable notions of regularities can be introduced.

An operative general dichotomy is in fact between \emph{regular maps}, where a certain behaviour of $T_k$, such as cooling (see \eqref{effective-cooling}) or convergence in $L^1$ to a finite measure (see \eqref{L1-convergence}), is imposed and
\emph{non-regular resampling maps} where $T_k$ may oscillate, remain bounded on
subsequences, or display non-monotone eventual growth.
Identifying the class
of resampling maps that give rise to a certain behaviour is part of what we refer to as the \emph{game of mass}, Sec. \ref{ssec:game_of_mass}. 
One message of the paper is that resampling maps naturally allow to interpolate between the homogenized or frozen extremes, corresponding respectively to maps with strong dominance of very frequent resampling or with asymptotically negligible  resettings. Still,
\emph{general resampling maps} permit very flexible perturbations: in turn specific choices of $\tau$ can produce a random walk $X(\mathfrak{X}, \tau)$ that behaves very differently from any underlying random walk $X^{(j)}$ that make up $\mathfrak{X}$.

\subsection{Asymptotic behaviours at glance: persistence, perturbation, and cross-over driven by resampling}
\label{s:phenomenology}
In the 1d setting, we will be mainly interested in the following questions that characterise the RWCRE asymptotics.

\begin{enumerate}
\item \textbf{Recurrence versus transience.}
Does $X_n$ visit the origin infinitely often, or does it drift away (to the
right or to the left)?
This is qualitative question,  presented in Section~\ref{sec:recurrencevstransience}, reveals a dependence on the resampling scheme which can be highlighted as follows.
\begin{itemize}
\item \emph{Persistence of transience under cooling.}
When the underlying environment law $\mu = \alpha^\bb{Z}$ produces a right-transient RWRE, then
right-transience persists for all genuinely cooling maps, i.e.\ as soon as \eqref{effective-cooling} holds true.
Thus, transience is persistent (non-perturbable) under cooling maps.

\item \emph{Perturbability of recurrence.}
When the underlying RWRE is recurrent, it is possible to build (exotic) cooling maps which results into a transient RWCRE. Therefore, recurrence is a perturbable behaviour, i.e., there is a cooling map $\tau^*$ for which the process  $X^* = (X_t(\mathfrak{X},\tau^*), t \geq 0)$ is transient, even when for every $j \in \bb{N}$ the underlying $X^{(j)}$ processes is a recurrent random walks. 
\end{itemize}

Moreover, as we shall see, also the transience property becomes perturbable as soon as one considers resampling maps that are ``cooling'' in a weaker (Cesàro) sense, see \eqref{Cesarocooling}.
\item \textbf{Speed and deviations.}
Does the average speed $\frac{X_n}{n}$ converge and, if so, what is the limiting speed? In Section \ref{sec:ergodic} we consider whether the limiting speed remains that of the static RWRE or can be altered by resampling and look at related deviations. That is, does $\frac{X_n}{n}$ satisfy a large deviation principle under the quenched/annealed measures? How does the rate function depend on $\tau$?

For these questions associated to laws of large numbers and large deviations, a guiding heuristic is that cooling, as in \eqref{effective-cooling}, should preserve the static RWRE behaviour, because the walk repeatedly experiences long frozen stretches. This heuristic is often correct for ``regular'' cooling maps, but it can fail for general resampling maps. In particular:
\begin{itemize}
\item under broad cooling conditions, LLN and (annealed/quenched) LDP
match those of the static RWRE;
\item under more general resampling schemes, one can observe a rich variety of
effective speeds and deviation behaviours, driven by how $\tau$ allocates mass to
different block lengths.
\end{itemize}

\item \textbf{Fluctuations: mixtures, robustness, and new limits.}
Under which centering and scaling, does $X_n$ converge to a non-degenerate random variable? How do those scaling limits compare with those of the corresponding RWRE?

These questions concerning fluctuations are the most sensitive and display the richest cross-over.
Let $\rho=\frac{1-\omega(0)}{\omega(0)}$ and recall the usual RWRE parameter $s > 0$
defined implicitly  for right transient RWRE by
\begin{equation}\label{eq:s_parameter}
\bb{E}[ \rho^s] = 1.
\end{equation}
With a slight abuse of notation we let $s=0$ represent the recurrent Sinai regime corresponding to $\bb{E}[\log\rho]=0$. We organise the results according to method and phenomenology.
The summary picture is:
\begin{itemize}
\item \textbf{$s  \in [0,1) \cup (2,\infty)$.} Cooling produces non-trivial sub-diffusive new behaviour governed by how the resampling map distributes variance
across blocks; regular maps yield explicit cross-over scalings and limit points. With centering by the mean and scaling by the standard deviations, the limit points are mixtures of the limit law obtained for the underlying static regime: Sinai-Kesten laws, when $s = 0$, second-kind Mittag-Leffler distributions law when  $s \in (0,1)$, and standard Gaussian random variables, whose mixtures remain standard Gaussian, when $s >2$.
\item \textbf{$s\in(1,2]$}. For these values of $s$ variance scaling is no longer possible and the analysis becomes more subtle. When $s=2$ the limit points are possibly non-standard Gaussian distributions. When $s \in (1,2)$,  new non-Gaussian limits arise, identified as generalized tempered stable laws and related mixtures. 
\end{itemize}
\end{enumerate}
The case when $s=1$ remains open.
\subsection{Methods: block decomposition, mass profiles, and replacement principles}
\label{s:methods}

The analysis throughout the survey is organized around the decomposition
\eqref{eq:X_dec} and the fact that $(Y_k =  X_{(\tau(k)\wedge t)}-X_{(\tau(k-1)\wedge t)}, k \in \bb{N})$ are independent.
This reduces many questions about RWCRE to questions about \emph{weighted sums of
independent increments}.

\subsubsection*{(i) Blockwise reduction and ``mass profiles''}
A convenient way to encode the influence of $\tau$ is to measure the relative
contribution of each block to the relevant scale (mean, variance, or cumulant).
For fluctuations, one typical choice is the variance profile
\begin{equation}\label{eq:lambda_profile_generic}
\lambda_{\tau,n}(k) :=\frac{\text{Var}\big(Y_k\big)}{\text{Var}\big(X_n\big)}, \qquad 1\le k\le \ell_n-1,
\end{equation}
together with the analogous boundary weight for $\bar Y^{\,n}$. Regularity conditions can then be formulated as asymptotic statements about $\lambda_{\tau,n}$ (e.g.\ whether the boundary weight vanishes, or whether the sequence admits a non-trivial limit along subsequences). This is precisely the mechanism behind the Sinai--Kesten mixture description: limit points are governed by the limiting ``mass distribution'' of $\lambda_{\tau,n}$.

\subsubsection*{(ii) Cooling ergodic principles and strong laws}
In this paper we hold the understanding that ergodic statements are statements about the time averages of the process $X_t$ that arise from mixing on the configuration space. Therefore even though the process $X_t$ does not fit the framework of an i.i.d process nor of a Markov process and may not display stationary behaviour on the path space, we consider the statements about $X_t/t$ to be of an ergodic nature. To distinguish it from the classical framework of ergodic theory, we call it \textbf{cooling ergodic principles}.

For LLN-type results, one combines:
(a) ergodic/renewal input from the frozen RWRE inside each block, and
(b) concentration or moment bounds ensuring that the blockwise averages combine
according to the proportions dictated by $\tau$.
This viewpoint naturally leads to ``ergodic theorems under resampling'' and to
general strong laws for weighted sums of independent variables regulated by an
increment sequence, as in the ``game of mass'' philosophy explained in Section \ref{ssec:game_of_mass} below.

\subsubsection*{(iii) Large deviations via cumulants and block additivity}
For LDP, the decomposition allows one to express log-moment generating functions
of $X_n$ as sums of block contributions (plus a boundary term), so that one can
apply Cram\'er-type arguments once suitable control on cumulants is available.
A recurring theme is to identify when the boundary block is negligible and when
it creates an additional contribution that changes the effective rate.

\subsubsection*{(iv) Replacement schemes and robust limit theorems}
In regimes where the frozen RWRE has an explicit scaling limit (Sinai--Kesten,
stable, Gaussian), one often aims to \emph{replace} each block increment by its
asymptotic limit at the corresponding block length and then study the resulting
weighted sum. This replacement method works when the underlying RWRE exhibits convergence in $L^2$, that is, when $s \in [0,1) \cup (2, \infty)$.
In the remaining regimes, $s\in(1,2]$, the replacement
idea still plays a role but must be supplemented by additional inputs, such as control of tail asymptotics, refined truncations, or regime-specific stable/tempered-stable
technology.

\section{Homogenization for non-cooling maps}\label{sec:homogenization}

In this section, as a warm-up, we treat a simple case of resampling maps $\tau$ among the class of non-cooling maps for which \eqref{effective-cooling} is violated.

The resampling maps that do not satisfy \eqref{effective-cooling} have $\liminf_k T_k <\infty$, and are therefore referred to as \emph{non-cooling maps}. Among these non-cooling maps homogenization is ensured, under regularity condition, and the main asymptotic results can be obtained by standard large deviation techniques through a careful inspection of the scaled cumulant generating function (s.c.g.f.).

We now define the empirical measure of the $m$ first increments by $\nu_m := \frac{1}{m} \sum_{k = 1}^m \delta_{T_k}$ and assume that $\nu_m \to \nu_*$ in $L^1$, that is,
\begin{equation}\label{L1-convergence}
    \lim_{m \to \infty} \sum_{T \in \bb{N}} T|\nu_m(T)- \nu_*(T)| = 0.
\end{equation}
We further assume that the piece increments are negligible when compared to the elapsed time
\begin{equation}\label{non-cooling-bounded}
    \lim_k \; \frac{T_k}{\tau(k)}  = 0.
\end{equation}

Under the regularity conditions \eqref{L1-convergence} and \eqref{non-cooling-bounded}, RWCRE behaves as a homogeneous random walk. To see this, consider the scaled cumulant generating function
\[
J^*_m(a):=\frac{1}{m} \log \bb{E}[\exp(aX_{\tau(m)})]
\]
and note that if \eqref{L1-convergence} holds true then, using the notation in \eqref{eq:space_incr} and the indepedent decomposition in \eqref{eq:X_dec} , we can write
\begin{equation}\label{cumulant-convergence}
\begin{aligned}
J^*_m(a) &=  \frac{1}{m} \sum_{j = 1}^m \log \bb{E}[\exp(a Y_j)] = \sum_{T \in \bb{N}} \log \bb{E}[\exp(a Z_T)] \nu_m(T) \\
&\to \sum_{T \in \bb{N}} \log \bb{E}[\exp(a Z_T)] \nu_*(T) =:J^*(a) \in \bb{R},
\end{aligned}
\end{equation}
where the limit passage above follows from \eqref{L1-convergence}, the dominated convergence theorem, and the fact that $|\log \bb{E}[\exp(a Z_T)]| \leq |aT|$.

Now, we note that $D_{J^*}: = \{a: J^*(a) <\infty\} = \bb{R}$, $a \mapsto J^*(a)$ is smooth,
and  strictly convex.  It follows from the classical from Gärtner–Ellis Theorem, see \cite[Thm V.6]{dH00}, that $(X_{\tau(j)}/j,\, j \in \bb{N}) $ satisfies the Large Deviation Principle (LDP) with rate $j$ and rate function 
\[
J(x) := \sup_{a } \big[ax -J^*(a)\big].
\]
To obtain the LDP for $X_n$ we relate $n$ and $\tau(n)$, by noting that \eqref{L1-convergence} implies
\[
\frac{\tau(n)}{n} = \sum_{T \in \bb{N}} T \nu_n(T) \to \sum_{T \in \bb{N}} T \nu_*(T) =:\bar{T}<\infty.
\]
By \eqref{non-cooling-bounded}, it follows that 
$\lim_n \bar{T}_n/n =0$ and thus $\lim_n \frac{1}{n} \log \bb{E}[\exp(a Z_{\bar{T}_n})] =0$. In view of \eqref{cumulant-convergence}, we obtain that
\begin{equation}\label{I-cumulant-def}
\begin{aligned}
\frac{1}{n} \log \bb{E}[\exp(a X_n)] & = \frac{1}{n} \log \bb{E}[\exp(a X_{\tau(\ell_n)})] + \frac{1}{n} \log \bb{E}[\exp(a Z_{\bar{T}_n})]\\
&= \frac{\ell_n}{n} J^*_{\ell_n}(a) +\frac{1}{n} \log \bb{E}[\exp(a Z_{\bar{T}_n})]\\
&\to \frac{1}{\bar{T} }J^*(a) =: I^*(a).
\end{aligned}
\end{equation}
Now, we apply again Gärtner–Ellis Theorem, \cite[Thm V.6]{dH00}, to conclude that $(X_{n}/n, n \in \bb{N}) $ satisfies the LDP with rate $n$ and rate function 
\[
I(x) := \sup_{\theta} \big[x\theta - I^*(\theta)\big].
\]

Once established such an LDP with 
rate function for $X=(X_n, n \in \bb{N})$, as detailed below, the strong LLN follows, and provided further regularity on the map $\tau$, a Central Limit Theorem (CLT) for $X$ as well. 

\paragraph{LLN from LDP}
Let 
\begin{equation}\label{v-def-I}
v := (I^*)'(0) = \sum_{T} \bb{E}[Z_T] \nu_*(T)/\bar{T}.
\end{equation}
Since $I^*$ is strictly convex, it follows that $v$ is the unique value for which $I(v) = 0$. With the help of the Chernoff-bound we deduce that for any $\varepsilon>0$
\[
\limsup_n \frac{1}{n} \log \bb{P}(X_n/n\geq v + \varepsilon)\leq  -I(v+\varepsilon) <0,
\]
which implies that $\sum_n \bb{P}(X_n/n\geq v + \varepsilon) < \infty$. Similarly, we obtain that $\sum_n \bb{P}(X_n/n\leq v - \varepsilon) < \infty$ and therefore, by the first Borel-Cantelli lemma, we can conclude the strong LLN. 
\paragraph{CLT from the LDP}
In order to derive a CLT we impose stricter conditions than \eqref{L1-convergence} and \eqref{non-cooling-bounded}.
We assume that convergence of measure occurs in $L^2$
\begin{equation}\label{eq:CLT-moment}
\lim_m \sum_{T\in\bb N} T^2|\nu_*(T)- \nu_m(T)| = 0, 
\end{equation}
and that the increment sequence is such that the boundary terms are negligible under diffusive scaling, that is
\begin{equation}\label{boundary-negligibility}
    \lim_k \frac{T_k}{\sqrt{\tau(k)}} = 0.
\end{equation}
In order to recover the classical centering by $nv$, we further assume 
that the speed converges in the diffusive scale, that is
\begin{equation}\label{eq:speed-for-clt-rate}
{\lim_m \sqrt{m}\sum_{T\in\bb N} T|\nu_*(T) -\nu_m(T)| = 0.}
\end{equation}
We derive the CLT in three steps, first we obtain the CLT for sums of centered blocks $\chi_j : = Y_j - \bb{E}[Y_j]$, next we adjust the time from block time/count to ``physical time'', i.e. basic step count, finally we recover the centering by $nv$.

\medskip
\noindent\emph{Step 1: CLT at block times.}
Set $\xi_j:=Y_j-\bb{E}[Y_j]$. By construction, see \eqref{eq:space_incr} it follows that $(\xi_j)_{j\ge1}$ are independent, centered, and $
\text{Var}(\xi_j)=\text{Var}(Z_{T_j})$.
Let $S_m:=\sum_{j=1}^m \xi_j$ and $s_m^2:=\text{Var}(S_m)=\sum_{j=1}^m \text{Var}(Z_{T_j})$.
By~\eqref{L1-convergence} and~\eqref{eq:CLT-moment},
\begin{equation}\label{eq:variance-growth}
\frac{s_m^2}{m}=\sum_{T\in\bb N}\text{Var}(Z_T)\,\nu_m(T)\ \rightarrow\ 
\sum_{T\in\bb N}\text{Var}(Z_T)\,\nu_*(T)=:\sigma^2_* = (J^*)''(0).
\end{equation}
Moreover, condition in~\eqref{boundary-negligibility} implies the Lindeberg condition for the triangular array
$\{\xi_j/s_m:1\le j\le m\}$, hence,  by the Lindeberg--Feller CLT, we obtain
\begin{equation}\label{eq:CLT-block-times}
\frac{S_m}{s_m}\ \Rightarrow\ \mathcal N(0,1),
\qquad\text{and by \eqref{eq:variance-growth}}\qquad
\frac{S_m }{\sqrt{m}}\ \Rightarrow \mathcal N(0,\sigma^2_*).
\end{equation}

\medskip
\noindent\emph{Step 2: from block times to physical time.}
First let $\sigma^2 :=\sigma^2_*/\bar{T}$ and note that $\sigma^2=(I^*)''(0)$. Now, with $\ell_n$ as in \eqref{eq:ell_def}, 
we can bound the remainder term 
\[
R_n:=X_n-X_{\tau(\ell_n)},\quad \text{with } |R_n|\le T_{\ell_n+1}.
\]
It then follows from \eqref{boundary-negligibility}  that $R_n/\sqrt{\tau(\ell_n)} \to 0$ and since $\tau(n)/n \to \bar{T}$ we obtain from \eqref{eq:CLT-block-times} that 
\begin{equation}\label{CLT-homogenisation}
\begin{aligned}
\frac{X_n - \bb{E}[X_n]}{\sqrt{n}} &= \frac{\sqrt{\ell_n}}{\sqrt{n}}\frac{X_{\tau(\ell_n)} - \bb{E}[X_{\tau(\ell_n)}]}{\sqrt{\ell_n}} + \frac{R_n - \bb{E}[R_n]}{\sqrt{n}} \\
&\to \frac{1}{\sqrt{\bar{T}}} \mc{N}(0, \sigma^2_*) = \mc{N}(0,  \sigma^2_*/\bar{T}) = \mc{N}(0,  \sigma^2).
\end{aligned}
\end{equation}

\medskip

\noindent\emph{Step 3: Centering and scaling constants and the LDP.}
It follows from \eqref{v-def-I} and \eqref{eq:CLT-moment} that 
$0=\sum_{T\in\bb N} (\bb{E}[Z_T]- vT)\nu_*(T)$.
Therefore we have that
\begin{equation}
\begin{aligned}
\frac{\bb{E}[X_{\tau(m)}]-v\,\tau(m)}{m}
&=\sum_{T\in\bb N} (\bb{E}[Z_T] vT) \nu_m(T)
\\
&=\sum_{T\in\bb N}\big(\bb E[Z_T]-vT\big)\big(\nu_m(T)-\nu_*(T)\big).
\end{aligned}
\end{equation}
{Therefore with the help of \eqref{eq:speed-for-clt-rate} we obtain that} 
\begin{equation}\label{eq:mean-drift-block}
\begin{aligned}
\frac{\bb E[X_{\tau(m)}]-v\,\tau(m)}{\sqrt{m}}
&= \sqrt{m}\sum_{T\in\bb N}\big(\bb E[Z_T]-vT\big)\big(\nu_m(T)-\nu_*(T)\big) \to 0.
\end{aligned}
\end{equation}
By \eqref{boundary-negligibility} it follows that
$\frac{\bb{E}[X_n] - n v}{\sqrt{n}} \to 0$ and therefore we obtain the CLT with centering by the speed:
\begin{equation}\label{CLT-homogenisation-1}
\begin{aligned}
\frac{X_n - nv}{\sqrt{n}} &\Rightarrow \mc{N}(0, \sigma^2).
\end{aligned}
\end{equation}

We may relate the centering term and the variance of the CLT with $v$, the minimum of $I$ and its curvature. 
Indeed, to obtain an explicit centering term of the CLT, we know from \eqref{v-def-I} that $v = (I^*)'(0) =  \sum_{T} \bb{E}[Z_T] \nu_*(T)/\bar{T}$ is the only point for which $I(v) = 0$. Moreover, $1/I''(v) = (I^*)''(0) = \sigma^2_*/\bar{T} = \sigma^2$.
This gives the centering and scaling terms. If we also note that $v>0$ implies that $X_n$ is right-transient, $v<0$ implies that $X_n$ is left-transient, and $v=0$ implies, with the help of the CLT result, that $X_n$ is recurrent. 

In summary from the inspection of the LDP rate function we have determined all the classical long-term statements of $X$. This justifies the ``folklore" that
\begin{equation}\label{Folklore}
    \text{the LDP for $X$ encodes the asymptotic behaviours of $X$.}
\end{equation}
\begin{remark}[On the CLT: Lindeberg--Feller vs.\ Bryc-type criteria]
In the present homogenization regime the CLT is most naturally obtained from the
independent block decomposition and the Lindeberg--Feller theorem for triangular arrays.
This yields, after centering by $v= (I^*)'(0)$ and scaling by the diffusive rate $n^{-1/2}$,  a Gaussian limit with variance $\sigma^2=(I^*)''(0)$.

It is worth noting that there exists an alternative ``Bryc-type'' route to a CLT, which is more
analytic than probabilistic: one first proves a second-order expansion of the scaled
log-moment generating functions
\[
I_n^*(\theta):=\frac1n\log \bb{E}[e^{\theta X_n}]
\]
in a neighborhood of $0$, with sufficient uniformity to evaluate $I_n^*(t/\sqrt n)$.
More precisely, if $I_n^*(\theta)\to I^*(\theta)$ for $|\theta|$ small and if
\[
n\Big(I_n^*\Big(\frac{t}{\sqrt n}\Big)-I^*\Big(\frac{t}{\sqrt n}\Big)\Big)\to 0
\quad\text{for each fixed $t$},
\]
then a standard moment generating function (m.g.f.) argument implies
\[
\frac{X_n-n(I^*)'(0)}{\sqrt n}\Rightarrow \mathcal N\big(0,(I^*)''(0)\big).
\]
In general, a rate $n$ LDP alone does not imply such a CLT; the additional local
uniformity above is the key extra input.
\end{remark}

In what follows we shall examine the long-term behaviours for $X_n$ for resampling maps that are not as simple as for those fulfilling the regularity conditions described in this section.
Special role will be given to cooling maps, see \eqref{effective-cooling},  and other non-regular variants (even with $\liminf_j T_j < \infty$) as opposed to the maps where  homogenization occurs or the static ``frozen behaviour" is in force, $\tau(1) =\infty$.
Interestingly, while for the 
resampling maps treated in this section all the results for $X$ can be deduced from the rate function, this approach is no longer possible outside this regular class ensuring homogenization as the corresponding rate function is no longer analytic and local uniformity does not hold. 

\section{Recurrence and transience properties in 1d}\label{sec:recurrencevstransience}

In this section, we discuss for 1d RWCRE how resampling maps
can affect the celebrated recurrence versus  criterion due to Solomon \cite{Sol75} for RWRE.
In fact, a general recurrence criterion for any resampling maps remains 
an open problem on the integers for various subtle reasons that we next highlight.

While a full recurrence/transience classification for arbitrary resampling maps is too broad to expect without additional structure, the RWCRE setting naturally suggests concrete and tractable open directions. A particularly appealing one is to identify the minimal cooling rate that still preserves the recurrence/transience dichotomy. 

Let us start with a few remarks. Since for any event $A$,
\begin{equation}
    \bb{P}(A)=1 \quad \iff \quad P_0^{\textcolor{black}{\bar{\omega}},\tau}(A)=1 \quad \mu^\bb{N}\text{-a.s}.,
\end{equation}
we do not distinguish between quenched and annealed statements when it comes to zero-one laws. In particular, despite the fact that RWCRE lack an invariant measure in path space and therefore is not \emph{ergodic} as defined in the standard dynamical system sense (see also Section~\ref{sec:ergodic}), it is still  \emph{tail-trivial}, i.e., all events in the tail sigma-algebra have probability zero or one. This is captured in the following lemma whose proof its standard. 
\begin{lemma}[Tail triviality for RWCRE]\label{lem:kolmogorov01-cooling}
Given a pair $(\alpha,\tau)$, let 
$$\mc{T} = \bigcap_{n\in \bb{N}} \sigma(X_m, m >n)$$ be the tail sigma algebra generated by the associated RWCRE $X$. If $A \in \mc{T}$ then $\bb{P}(A) \in \{0,1\}$.
\end{lemma}

We know from \cite{Sol75} that by analysing the related quenched potential
RWRE is recurrent if and only if $\bb{E}[\log\rho] = 0$ and right-transient if and only if $\bb{E}[\log\rho] < 0$.  We thus  say that $\alpha$ is \emph{recurrent} or \emph{right-transient} when $\bb{E}[\log\rho] = 0$, respectively, $\bb{E}[\log\rho]< 0$. For RWCRE the classification of recurrence versus transience is more delicate, because it also depends on the resampling map $\tau$ and therefore, under the quenched measure $P^{\bar{\omega},\tau}$ the Markov process $X_n$ is not time-homogeneous. 
In what follows we say that $(\alpha,\tau)$ is \emph{recurrent} or \emph{transient}, respectively, when
\begin{equation}
\bb{P}(X_n=0 \,\, \text{i.o.})=1 \quad \text{or} \quad \bb{P}(X_n=0 \,\, \text{i.o.})=0.
\end{equation}
We say that $(\alpha,\tau)$ is \emph{right transient} or \emph{left transient} when
\begin{equation}
\bb{P}\left(\liminf_{n\to \infty} X_n=\infty \right)=1 \quad \text{or} \quad \bb{P}\left(\limsup_{n\to \infty} X_n=-\infty \right)=1.
\end{equation}
By tail triviality, $\{0,1\}$ are the only possible values for the above events.

A key feature in the study of recurrence and transience for RWCRE is the analysis of the drift over the time increments  $\bb{E}[X_{\tau(j)}-X_{\tau(j-1)}]$ for $j \in \bb{N}$. While, for RWRE, it is the quantity $\bb{E}[\log \rho]$ 
rather then the local drift that determines the criterion allowing for cases in which e.g. the drift at a particular point can be positive and the process can still be left-transient.

The main general result as far as recurrence of the RWCRE is concerned is captured by the following theorem, derived in \cite[Thm. 1]{AveChidCosdHol22}. Restricted to cooling maps as defined in \eqref{effective-cooling}, the theorem offers subtle sufficient conditions under which the Solomon criteria are not affected by the resettings due to $\tau$.

 We say that a measure $\alpha$ with support in $(0,1)$ is \emph{symmetric} when for all $a,b \in (0,1)$, $\alpha([a,b]) = \alpha([1-b,1-a])$, that is $\alpha$ is symmetric around $1/2$.
The key general statement on recurrence for RWCRE is \cite[Thm.~1.7]{AveChidCosdHol22}. In the cooling regime \eqref{effective-cooling}, it provides sufficient conditions under which resampling along~$\tau$ does not alter the Solomon criterion.

\begin{theorem}[{\bf Stability of recurrence/transience, \cite[Thm~1.7]{AveChidCosdHol22}}] \label{thm:RvT}
\text{}\\ \vspace{-.5cm}
\begin{itemize}
\item[{\rm(a)}] If $\alpha$ is right-transient, then $(\alpha,\tau)$ is right-transient for all cooling maps $\tau$, that is, such that \eqref{effective-cooling} is valid.

\item[{\rm(b)}] If $\alpha$ is recurrent, then $(\alpha,\tau)$ is recurrent when
\begin{equation} \label{sufficient_rec}
\liminf_{n\to\infty} \abs{\bb{E} \left[ \frac{X_n}{\sqrt{\mathsf{Var} (X_n)}} \right]} = 0.
\end{equation}
The latter holds for all symmetric $\alpha$ and all $\tau$,
and also for all non-symmetric $\alpha$ when $\tau$ is such that
\begin{equation}\label{sufficient_fast}
\liminf_{k\to\infty} \frac{1}{k^\gamma} \log T_k >0 \text{ for some } \gamma > \frac34.
\end{equation}
\end{itemize}
\end{theorem}


As anticipated, a complete recurrence criterion for {general resampling maps}, or even for general cooling maps, is still lacking. To {illustrate} the subtleties involved, we refer the reader to the discussion and counterexamples in \cite{AveChidCosdHol22}{,} see Ex.1 and Ex.2 therein. {Their main messages may be summarised by the following two bullets, which highlight the limitations of Theorem~\ref{thm:RvT}(a) and Theorem~\ref{thm:RvT}(b), respectively.}

\begin{itemize}
\item\emph{Right-transient can turn into left-transient or recurrent}:
There exist a right-transient environment law $\alpha$ and two resampling maps $\tau'=\tau'(\alpha)$ and $\tau''=\tau''(\alpha)$ whose increment sequences are Cesàro-divergent, in the sense that
\begin{equation}\label{Cesarocooling}
\lim_{\ell\to\infty}\frac{1}{\ell}\sum_{k=1}^{\ell} T_k=\infty,
\end{equation}
while nevertheless $(\alpha,\tau')$ is left-transient and $(\alpha,\tau'')$ is recurrent.

\item \emph{Recurrent can turn into transient}:
There exist a recurrent $\alpha$ and a cooling map $\tau=\tau(\alpha)$ such that \eqref{sufficient_fast} is not valid,
for which $(\alpha,\tau)$ is transient.
\end{itemize}

We conclude this section with a few remarks concerning the methods used in \cite{AveChidCosdHol22} to derive Theorem \ref{thm:RvT} above when compared with the methods used in the proof of the classical statements for RWRE. The proof of the latter presented in \cite{Sol75} is obtained from the analysis of a quenched harmonic function associated to the underlying birth and death chain,
but it is worth mentioning that it can be reframed in the semi-martingale approach of  \cite{FayMalMen95}. 

In RWCRE the quenched law $P^{\bar\omega,\tau}$ is time-inhomogeneous due to resampling at times $\tau(k)$. Within block $k$ the walk evolves as an ordinary RWRE in the fixed environment $\omega_k$, and the corresponding quenched scale function $S_{\omega_k}$ 
is harmonic for the within-block kernel.
At a refreshing time, however, the kernel changes abruptly and the harmonicity relation for $S_{\omega_k}$ is in
general incompatible with the next block kernel. 
For this reason, the recurrence/transience results for RWCRE in \cite{AveChidCosdHol22} are proved by exploiting the decomposition
at refreshing times: the displacement over block $k$ is an RWRE increment in environment $\omega_k$, and under the
annealed law the sequence of block increments is tractable. This block structure makes it possible to prove, for instance, stability of directional transience under \emph{cooling} $T_k\to\infty$:
long blocks are close to the ``unperturbed'' RWRE behaviour, so the effect of refreshing becomes asymptotically
negligible. In contrast, when $T_k$ does not diverge (or only diverges in a weaker Ces\`aro sense as in \eqref{Cesarocooling}), the refreshing
mechanism can genuinely alter the recurrence classification, and the appropriate tools become annealed
drift/ergodic/concentration estimates for block displacements rather than a single quenched harmonic function.

 \section{Cooling Ergodic tools for LLN \& LDP}\label{sec:ergodic}
This Section treats the law of large numbers (LLN) and large deviation principle (LDP) for the average displacement  side by side because, in
dynamic/resampled environments, the same block decomposition and the same scale control both. In RWCRE, one naturally decomposes the path at the resampling times $\tau(k)$, so that additive observables (such as displacement, log-moment generating functions)
become weighted sums of block contributions plus a boundary term.
Once this decomposition is in place, essentially the same arguments that yield almost sure control of block averages
also yield almost sure convergence of block log-moment generating functions; from there, an LDP often comes
``as a bonus'' of the approach.

At the same time, there is an opposite (and equally instructive) viewpoint. A common piece of ``folklore'', recall \eqref{Folklore}, is that an LDP, when sufficiently regular, encodes much finer information than the LLN: it prescribes not only the limiting speed but also the exponential cost of deviations and---under additional smoothness and uniformity assumptions---can even imply fluctuation statements (e.g.\ via Bryc-type arguments connecting Laplace asymptotics to central limit behaviour). In that sense, one may start from large deviations as the ``master object'' and see the LLN as its most visible
corollary: the LLN would correspond to the unique minimiser of the rate function. 
For RWRE, this is only the case in the Sinai/recurrent and the transient sub-ballistic regimes, corresponding to ``$s=0$" and $s\in(0,1)$, respectively, with $s$ as in \eqref{eq:s_parameter}. In the positive speed regime $s>1$, there is no hope to recover the fluctuations due to the flat piece of the rate function $I$, $I(u) = 0$ for $u \in [0,v]$, which forbids us to apply LDP estimates to derive the LLN and makes the rate function non-smooth in violation of the conditions required for Bryc's CLT arguments to hold. Still, for RWCRE under the Sinai regime, given that there is a cooling map with $T_k \to \infty$ for which the fluctuation is Gaussian, one could hope to capture the effect of the cooling map in the LDP.

This dual narrative is useful here because it clarifies what our ergodic techniques are actually producing:
they are not only proving a strong law, but also constructing the analytic object behind the LDP, namely the limiting
scaled cumulant generating function (s.c.g.f.)
\begin{equation}\label{eq:ch4-intro-scgf}
\Lambda(\theta)\;=\;\lim_{n\to\infty}\frac{1}{n}\log \bb{E}\left[e^{\theta X_n}\right],
\end{equation}
whenever the limit exists in the relevant sense (quenched or annealed). Once $\Lambda$ is identified and has
appropriate regularity, the candidate rate function is its Legendre--Fenchel transform
\begin{equation}\label{eq:ch4-intro-legendre}
I(x)\;=\;\sup_{\theta\in\bb{R}}\{\theta x-\Lambda(\theta)\}.
\end{equation}
In Appendix \ref{app:LDP} we show a proof of LDP does not go through a standard inversion argument. Indeed, the proof we present follows is more artisanal, in the sense that it controls error terms and proves the upper and lower bound one at the time.

\textbf{Quenched Vs Annealed analysis.} A second reason to treat LLN and LDP together is that beyond the existence of two LLNs (the weak and the strong) we can consider, and must distinguish two
large deviation principles.
In a dynamic random environment there is ``randomness of the walk given the environment'' and ``randomness of the environment itself''.
The \emph{quenched} viewpoint conditions on the whole environment (and, in RWCRE, on the entire resampling schedule
and the refreshed environments), proving almost sure statements for typical realisations.
The \emph{annealed} viewpoint averages over the environment, producing a generally different large deviation
principle, with a potentially different rate function because rare deviations of the walk can be helped (or hindered)
by atypical environments.

The role of \emph{concentration} depends strongly on which limit one seeks.
In our setting, concentration is a central ingredient for the \emph{strong} LLN and for the \emph{quenched} LDP:
it is the bridge that upgrades convergence in expectation (or in probability) into almost sure statements,
and it controls the errors in block-by-block Laplace asymptotics along a fixed environment realisation.
By contrast, the weak law of large numbers and the annealed LDP can be obtained by softer arguments,
since averaging over the environment
allows one to work directly at the level of expectations.
For completeness, we provide in this Section a proof of the annealed LDP, a new result in this literature.

It is worth emphasising that, in many models, the annealed LDP is \emph{harder} than the quenched LDP.
Heuristically, annealing gives the system the freedom to perform a deviation not only by an unusual trajectory
but also by sampling an unusual environment, and the optimisation typically involves an entropic cost
(through a relative-entropy term) balancing ``unlikely displacement'' against ``unlikely environment''.
Remarkably, RWCRE exhibits the opposite behaviour: the annealed LDP turns out to be simpler than one might expect,
essentially because the resampling structure decouples the problem into independent blocks under the annealed law.

Finally, it is instructive to place these results in historical perspective.
For one-dimensional RWRE, quenched large deviations were established first, in the work of Greven and
den Hollander \cite{GdH94}, while the annealed large deviations were obtained later through a sequence of works
by Zeitouni, Dembo, Peres, Comets, Gantert \cite{CGZ00} \cite{DPZ96} in the late 1990s and early 2000.
This chronology reflects the general phenomenon mentioned above: annealing introduces an additional optimisation over environments, and the corresponding rate function may contain a non-trivial entropic contribution.

We also contrast our approach with the subadditive ergodic theorems that are effective in stationary settings, with the Liggett--Kingman framework, \cite{Lig85}, used in RWRE by Ramirez and others, see \cite{CDRRS13}.
In resampled environments, stationarity is fundamentally broken: the ``environment seen from the particle'' changes
law at the refreshing times, making the process time-inhomogeneous.
For this reason, subadditive ergodic theorems do not apply directly, and one must replace them by a ``cooling'' ergodic
theorem tailored to weighted block concatenations.
The Section is devoted to developing precisely this replacement: ergodic limits for block sums and for block Laplace
transforms, robust enough to deliver the SLLN and both quenched and annealed LDP.

\textbf{Outline of the Section:} We start in Section \ref{WLLNannLDP} proving that for cooling maps $\tau$ the asymptotic speed and the associated annealed LDP of rate $n$ is the same as for RWRE, see Proposition \ref{prop:cooling-scgf}. In reality, we state general criterion see Theorem \ref{annLDPthm} which ensures when annealed LDP holds for RWCRE, in particular the persistence of the rate function in Proposition \ref{prop:cooling-scgf} and the LDP treated in Section \ref{sec:homogenization} can be seen as corollaries of the general Theorem \ref{annLDPthm}. We then move beyond the RWCRE specific framework and discuss in Section \ref{sec:gradual-sums} cooling ergodic theorems, first derived in \cite{AveChidCosdHol19,AvedCos24game}, to prove generic weak and strong convergence for the emerging weighted incremental sum, see \eqref{gradual}, in case of \emph{centred} random variables. For the RWCRE model, these results can on the one hand be used to show persistence of the limiting speed and of the quenched rate $n$ LDP of the static RWRE when $\tau$ is cooling as in \eqref{effective-cooling}, on the other hand, they can be used to analyse the emerging asymptotic speed for more general resampling maps, for which the above mentioned persistence is not guaranteed and the existence and characterization of the corresponding limiting value depends on the mean behaviour and the stability analysis of the chosen map $\tau$. The latter delicate problem is genuinely related to non-centred variables and is an instance of what we refer to as the game of mass which we discuss in Section \ref{ssec:game_of_mass}.

\subsection{Weak LLN for cooling maps and general annealed LDP}
\label{WLLNannLDP} 
When the considered resampling map $\tau$ is cooling as in \eqref{effective-cooling}, as anticipated, the empirical speed $X_n/n$ converges in probability (WLLN) and almost surely (SLLN) to the same limit as for RWRE. The proof of this convergence \emph{in probability} was first given in \cite[Sec. 3]{AdH17springer} by as soft argument, which we briefly recall, based on Markov's inequality and a so-called Toepliz-lemma. 
First we need some notation. For a fixed $\tau$ and underlying family $\mathfrak{X} = (X^{(j)}, j \in \bb{N})$ of i.i.d copies of annealed RWRE, let $X_t = X_t(\mathfrak{X},\tau)$ be the process defined by \eqref{cooling-general-tau}, let $v_T = \bb{E}[X^{(1)}_T/T]$, and let
\begin{equation}\label{gamma-def}
\gamma_{k,n}: = \frac{T_{k,n}}{n}, \quad \text{ where } T_{k,n} := (\tau(k) \wedge n) - (\tau(k-1) \wedge n).
\end{equation}
Now by the LLN for RWRE  we have that $v_T \to v_*$. For $\varepsilon>0$ there is $A = A(\varepsilon)$ such that $T>A$ implies $|v_n - v_*|<\varepsilon$.
Note that $\sum_k \gamma_{k,n} = 1$ and that \eqref{effective-cooling} implies that there is $n_0$ such that $\sum_{k} \bbm{1}_{\{T_{k,n}<A\}} \gamma_{k,n} < \varepsilon$ for $n>n_0$. Since $|v_T - v_*|\leq 2$ it follows from  Markov's inequality that 
\begin{equation}\label{toepliz-here}
    \begin{aligned}
        \bb{P}(|X_n  - \bb{E}[X_n] |>n\varepsilon)&\leq \frac{\bb{E}\big[|X_n - \bb{E}[X_n]|\big]}{n\varepsilon}
        \leq \sum_{k = 1}^{\ell} \gamma_{k,n} |v_{T_{k,n} - v_*}|\\
        &\leq \sum_{k} \bbm{1}_{T_k \leq A} 2\gamma_{k,n}  +  \sum_{k} \bbm{1}_{T_k > A} \gamma_{k,n} \varepsilon < 3 \varepsilon
    \end{aligned}
\end{equation}
Since $\varepsilon>0$ is arbitrary, and $\bb{E}[X_n]/n \to v_*$ it follows that $X_n/n$ converges in probability to $v_*$.

To prove the SLLN, namely, the same statement almost surely rather than in probability, the argument is more delicate 
and it is discussed in Section \ref{sec:gradual-sums}. 
In the next theorem, we prove an annealed LDP for the empirical speed under general resampling maps. As a consequence, Corollary \ref{prop:cooling-scgf} shows that, when one specialises to cooling maps, RWCRE exhibits the same annealed exponential-scale slowdown behaviour as the underlying RWRE.


\begin{theorem}[{\bf Annealed rate $n$ LDP for general resampling maps}]
\label{annLDPthm}
Let $\tau$ be a resampling map, let $X = (X_n, n \in \bb{N})$ follow the law of the annealed RWCRE defined in \eqref{eq:annealed_rwcre}, and let $\mu_n$ be the probability measure on $\bb{R}$ defined by $\mu_n(A): = \bb{P}(X_n/n \in A)$. 

If there is a function $\Lambda: \bb{R}\to \bb{R}$ that is strictly convex and such that
for all $\theta \in \bb{R}$,
\begin{equation}\label{scgf-convergence}
    \begin{aligned}
        \lim_n \frac{1}{n} \log \bb{E}[e^{\theta X_n}]&= \lim_n \sum_k \frac{T_{k,n}}{n} \frac{1}{T_{k,n}} \log \bb{E}[e^{\theta Z_{T_{k,n}}}] = \Lambda(\theta).
    \end{aligned}
\end{equation}
Then, $(\mu_n, n \in \bb{N})$ satisfies the large deviation principle in the sense of Definition III.6 of \cite{dH00} with rate $n$ and function $I$ given by 
\begin{equation}\label{rate_function}
I(x) = \sup\{x \theta - \Lambda(\theta):   \theta \in \bb{R}\}.
\end{equation}
\end{theorem}

The proof of Theorem \ref{annLDPthm} is given in appendix \ref{app:LDP}. The
key step  in the proof of Theorem \ref{annLDPthm} is the proof of convergence of the s.c.g.f. as given in the condition \eqref{scgf-convergence}. 
As a consequence of this theorem, one recovers the cooling-map case: when $\tau$ satisfies \eqref{effective-cooling}, the limit in \eqref{scgf-convergence} holds with $\Lambda(\theta)=\sup\{\theta x-I(x), x\in\bb{R}\}$, where $I$ is the annealed rate function for the empirical speed of the underlying static RWRE, as originally identified in \cite{GdH94}.

\begin{corollary}[{\bf Persistence of annealed LDP for cooling maps}]\label{prop:cooling-scgf}
{Consider a cooling map $\tau$ as in \eqref{effective-cooling} and an environment law $\mu$ as in \eqref{eq:mu_iid}. Let $X = (X_n, n \in \bb{Z}_+)$ be the annealed RWCRE obtained from $P^{ \mu, \tau}_0$ as defined in \eqref{eq:annealed_rwcre},} and for each $\theta \in \bb{R}$ let $I^*(\theta): = \lim_{T \to \infty} \frac{1}{T} \log \bb{E}[e^{\theta Z_T}]$, where $Z = (Z_n, n \in \bb{Z}_+)$ is the RWRE obtained from $P^\mu$ as defined in \eqref{eq:annealed_rwre}. Then  
    \begin{equation}\label{scgf-cooling-1}
    \begin{aligned}
        \lim_n \frac{1}{n} \log \bb{E}[e^{\theta X_n}]&= \lim_n \sum_k \frac{T_{k,n}}{n} \frac{1}{T_{k,n}} \log \bb{E}[e^{\theta Z_{T_{k,n}}}] = I^*(\theta).
    \end{aligned}
\end{equation}
{In particular, under the annealed measures, the empirical speeds of RWRE and RWCRE satisfy the same rate $n$ LDP with rate function given by the Legendre-Fenchel transform of $I^*$.}
\end{corollary}
\begin{proof}

{Recall the definition of $T_{k,n}$ given in \eqref{gamma-def}.
$X$ 
satisfies  \eqref{eq:increment_identification}, and in particular  $E_0^{\bar{\omega},\tau}\big[e^{\theta Y_k} \big]=E_0^{\omega_k}\big[e^{\theta Z_{T_k}}\big]$. The limit of the s.c.g.f. in \eqref{scgf-cooling} follows by applying the Toeplitz Lemma \cite[Lemma 1]{AdH17springer} in the same way as in \eqref{toepliz-here}. The conclusion then follows from  
Theorem \ref{annLDPthm}.}
\end{proof}


It is worth stressing that Theorem \ref{annLDPthm} also determines the LDP rate function for resampling maps that do not satisfy \eqref{effective-cooling}, including those from the simpler ``homogenising" class considered in Section \ref{sec:homogenization}. As we shall see in Section \ref{ssec:game_of_mass}, the convergence of $v_T \to v$ for the weak LLN discussed above, as well as the convergence of the s.c.g.f in \eqref{scgf-convergence} for the above annealed LDP is crucial in the proof of convergence of cooling ergodic averages of non-centred random variables. In the next section we discuss convergence of such general cooling ergodic averages starting with the case of centred random variables.
\subsection{Cooling Ergodic theorems for centered variables.}
\label{sec:gradual-sums}

Having established the annealed WLLN (at least for cooling maps) and a general annealed LDP for RWCRE, we now isolate the mechanism behind these results in an abstract form. The common structure is that relevant time-averaged observables decompose along refreshing blocks and, after normalization by the total time, become \emph{weighted sums} of independent block contributions. This leads us to study \emph{gradual sums}, i.e.\ convex combinations of independent random variables with weights dictated by the resampling map. Recall $\bar{T}_n = n - \tau(\ell_n)$, see \eqref{eq:time_incr}. For the study of the speed, we consider sums of the form
\begin{equation}\label{gradual}
\frac{X_n}{n}= \sum_{j=1}^{\ell_n-1} \frac{T_j}{n} \frac{X^{(k)}_{(\tau(j)\wedge n)-(\tau(j-1)\wedge n)}}{T_j} + 
\frac{\bar{T}_n}{n} X^{(\ell_n)}(\bar{T}_n) = \sum_j \gamma_{j,n} Y_{j,n}, 
\end{equation}
where the weights $\gamma_t = (\gamma_{j,t}, j \in \bb{N})$ satisfy $\sum_j \gamma_{j,t} = 1$ for all $n$ and the random variables $Y = (Y_{j,n}, j \in \bb{N})$ satisfy 
\[
Y_{j,n} = 
\begin{cases}
    X^{(j)}_{n \gamma_{j,n}}/(n \gamma_{j,n}), & \text { if } \gamma_{j,n}>0\\
0 &\text{ if } \gamma_{j,n} = 0.
\end{cases}
\]
The weights encode the relative time spent in each block up to time $n$, while the $Y_{j,n}$ encode the normalized block contributions. 

In general, for $t \in \bb{R}_+$,  we consider triangular arrays of weights $\gamma_t=(\gamma_{j,t})_{j\ge 1}$ with $\gamma_{j,n}\ge 0$ and $\sum_{j\ge 1}\gamma_{j,n}=1$, and triangular arrays of independent random variables $Y_n=(Y_{j,n})_{j\ge 1}$ 
The object of study is 
\begin{equation}\label{gradual-1}
\frac{X_t}{t} := \sum_{j\geq 1}\gamma_{j,t}\,Y_{j,t}.
\end{equation}

The purpose of this section is twofold. First, we introduce a general set of assumptions on $(\gamma_{j,t}, j \in \bb{N}, t \in \bb{R}_+)$ and on the concentration/tail behaviour of $(Y_{j,t})$ that guarantee a WLLN for $X_t/t$. Second, we identify additional hypotheses under which this can be strengthened to an SLLN. These results will serve as the abstract ``cooling ergodic tools'' used throughout the sequel and will later be specialized back to RWCRE to derive the SLLN  and the quenched LDP for the average displacement.

We have applied a generalization of Toepliz lemma, in the formulation of \cite[Lemma 1]{AdH17springer}, twice to obtain the WLLN and to determine the s.c.g.f. for the displacement of the annealed RWCRE. In both aplications we are dealing with sums of  the form $\sum_{k} \gamma_{k,n} z_{n \gamma_{k,z}}$ with $\gamma_{k,n} \geq 0$, $\sum_k \gamma_{k,n} = 1$, and $z_{T} \to z$ whenever $T \to \infty$. To focus on the noisy behaviour of these ergodic averages, we shall focus first on the study of centred random variables. We abstract from the set up of weighted sums and focus on the conditions on $\mathfrak{X}$ required to ensure that $X_t/t$ satisfies the WLLN when $X_t$ is as given in \eqref{cooling-general-tau}.
\begin{theorem}[{\bf Weak LLN, \cite[Thm~2.1]{AvedCos24game}}]
\label{incrementalweak} 
Assume that $\mathfrak{X}$ satisfies:
\begin{description}
	\item[\namedlabel{C}{C}] \emph{(Centering)}
\begin{equation}\notag
\forall\,m \in \bb{R}_+, \, k \in \bb{N};\quad \bb{E}\crt{X^{(k)}_m} = 0.
\end{equation}
\item[\namedlabel{W1}{W1}]  \emph{(Concentration)} 
\begin{equation}\notag
	\lim_{m\to\infty}    \sup_{k} \bb{P}\prt{\abs{X^{(k)}_m/m}>\epsilon} =0,
	\quad  \forall \varepsilon>0. 
\end{equation}   
\item[\namedlabel{W2}{W2}]  \emph{(Uniform Integrability)} 
 \begin{equation}\notag
	\lim_{A\to \infty} \sup_{k,m} \bb{E} \crt{\abs{X_k(m)}\Ind{\abs{X_k(m)}>A}} = 0. 
\end{equation}
\end{description}

Let $X_t = X_t(\mathfrak{X},\tau)$ be given by  \eqref{cooling-general-tau}. Then,
\begin{equation}\notag
\lim_{t\to\infty}\bb{P}\prt{\abs{X_t/t}>\epsilon}=0,\quad  \forall \varepsilon>0. 
\end{equation} 	
\end{theorem}
\begin{proof}
    See Section 5 in \cite{AvedCos24game}. 
\end{proof}
In order to move from the WLLN to the SLLN for sums of the form \eqref{gradual} one faces a challenge that is best explained by the following example. 

Let $\chv{U_k, k \in \bb{N}}$ be a sequence of
i.i.d. uniform random variables on $(0,1)$ and let
$\mathfrak{X} = (X^{(k)}_m := m g_m(U_k), k, m \in \bb{N})$ where
\begin{equation}\label{e:nce}
g_m(x) :=
\begin{cases}
 \phantom{-}1  & \text{if }   x \in (0,\frac{1}{2\log_2 (1 + m)}),\\
 -1 & \text{if }   x \in [\frac{1}{2\log_2 m},\frac{1}{\log_2 (1 +m)}),\\
  \phantom{-}0 & \text{else}.\\
\end{cases}
\end{equation}
Note that $\mathfrak{X}$ fulfills
assumptions~\eqref{C}, \eqref{W1}, \eqref{W2}.
Moreover, we have
\begin{equation}\label{e:ncr}
\bb{P}\prt{X^{(k)}_m/m= 1} = \frac{1}{2\log_2m},\quad \text{ and }
\quad \bb{P}\prt{X^{(k)}_m/m = -1} = \frac{1}{2\log_2m}.
\end{equation}
Now take $\tau$ such that$\tau_k = 4^k$. In this case we see that the incremental sum \eqref{gradual} does not
satisfy the strong LLN. Indeed, as
\begin{equation}\label{e:bc2}
  \sum_{k = 1}^\infty \bb{P}\prt{X^{(k)}_{T_k}/{T_k}=1} = \infty,\quad \text{ and } \quad
  \sum_{k = 1}^\infty \bb{P}\prt{X^{(k)}_{T_k}/{T_k}=-1} = \infty
\end{equation} 
by the second Borel-Cantelli lemma,
\begin{equation}\label{e:bc3}
\bb{P}\prt{X^{(k)}_{T_k}/{T_k}=1, \text{ i.o} } = 1,\quad \text{ and } \quad \bb{P}\prt{X^{(k)}_{T_k}/{T_k}=-1, \text{ i.o} } = 1.
\end{equation}
Note that $T_k = (4^{k} - 4^{k-1})$ and that by  \eqref{gradual}
\begin{equation}
\label{difcompute}
\begin{aligned}
\abs{\frac{X_{\tau(k)}}{\tau(k)} - \frac{X^{(k)}_{T_k}}{T_k}} &= \abs{\frac{X_{\tau(k)} - X^{(k)}_{T_k}}{\tau(k)} +X^{(k)}_{T_k}\Big(\frac{1}{\tau(k)} -\frac{1}{T_k}\Big)} \\
&\leq   \frac{\tau(k-1)}{\tau(k)} + \bigg(1 - \frac{T_k}{4^{k}}\bigg)\leq 2 \frac{4^{k-1}}{4^k}\leq \frac{1}{2}.
\end{aligned}
\end{equation}
Therefore
\[
\bb{P}(\abs{X_t/t - 1} < \frac{1}{2} \text{ i.o.}) =1,\quad
\text{ and } \quad
\bb{P}(\abs{X_t/t + 1} < \frac{1}{2} \text{ i.o.}) =1,
\]
which  means that almost surely $X_t/t$ does not converge.
\medskip


{In view of this example, to obtain a strong LLN in the centred case we impose further
conditions on $\bb{X}$.  In particular the concentration condition in Theorem~\ref{incrementalweak}
will be strengthened by requiring a mild polynomial decay and the
uniform domination.}

\begin{theorem}[{\bf Strong LLN, \cite[Thm~2.2]{AvedCos24game}}]
\label{incrementalstrong} 
Assume that $\bb{X}$ satisfies \eqref{C} and 
\begin{description}
\item[\namedlabel{S1}{S1}]\emph{(Polynomial decay)} There is a
  $\delta>0$ such that for all $\epsilon>0$ there is a $C = C(\epsilon)$ for
  which
\begin{equation}\notag
\sup_k\bb{P}\left(\abs{X^{(k)}_m/m}>\epsilon\right)<\frac{C}{m^\delta}.
\end{equation}
\item[\namedlabel{S2}{S2}] \emph{(Uniform domination)} There is a random variable $X_*$ and $\gamma>0$ such  that $\bb{E}(\abs{X_*}^{2 +\gamma })<\infty$ and for all $x \in \bb{R}$
\begin{equation}\notag
\sup_{k,m}\bb{P}(|X^{(k)}_m/m|)>x)\leq \bb{P}(X_*>x). \end{equation}  
\end{description}

Let $X_t = X_t(\mathfrak{X},\tau)$ be given by  \eqref{cooling-general-tau}. Then,
\begin{equation}\label{sLLNincremental}
\bb{P}\prt{\lim_{t\to\infty}X_t/t= 0} = 1.
\end{equation} 
\end{theorem}
\begin{proof}
    See Section 6 in \cite{AvedCos24game}. 
\end{proof}

In light of the above example, we see that the condition~\ref{S1} is
near to optimal.  Indeed, to improve it, we would need to find $T \mapsto f(T)$
satisfying
\[
\log^k(T)<<f(T) << T^\delta\quad \forall \, k \in \bb{N}, \; \delta >0 .
\]

\begin{remark}[{\bf Persistence of quenched rate function for cooling maps}]

{ One of the fruit of the above SLLN is that, analogously to the annealed case discussed in Section \ref{WLLNannLDP}, it can be used to establish a rate n quenched LDP for the RWCRE empirical speed, and when considering cooling maps 
as in \eqref{effective-cooling}, the resulting quenched rate function coincides with the rate function originally identified in \cite{GdH94} for the RWRE perturbed companion. Such a quenched persistent decay behaviour was first derived in \cite{AveChidCosdHol19} applying the SLLN above\footnote{Even if in 
\cite{AveChidCosdHol19} the statement was presented under  slightly less general conditions, see \cite[Thm~1.12]{AveChidCosdHol19}} 
to guarantee the almost sure existence of the limit in \eqref{eq:ch4-intro-scgf}, together with  
a key concentration result derived for RWRE,
see \cite[Thm~1.13]{AveChidCosdHol19}, which allows to check that condition~\ref{S1} above is satisfied and in particular permitting to invert the Legendre transform to go from the s.c.g.f to the rate function. We refer the reader to \cite{AveChidCosdHol19} for more details about this quenched rate function persistence, and why other general approaches for dynamic random environments such as those in \cite{CDRRS13} based on versions of the Kingmann-subadditive ergodic theorem as discussed e.g. in \cite{Lig85}, do not apply for the RWCRE model due to the lack of standard ergodicity.} \end{remark}

\subsection{The game of mass and explicit speed for general resettings}
\label{ssec:game_of_mass}

The random variables that we consider when dealing with sums of the form \eqref{cooling-general-tau} are not necessarily centred. If the conditions of Theorem \ref{incrementalstrong} are satisfied then the almost sure limit to $0$ of $(X_t - \bb{E}[X_t])/t$ is ensured. The limit of the sequence $X_t/t$ will then depend on the behaviour of $\bb{E}[X_t]/t$. The original set up we for which the ``centered'' cooling ergodic theorems from the previous section were developed assume that the limit of $\bb{E}[X_t]/t$ can be determined both in the study of LLN and LDPs (quenched and annealed). 

Naturally, the original set up is not the only ones for which a limit for $\bb{E}[X_t]/t$ holds true and the framework developed here allows for further generalization and abstraction. For example, one may consider alternating increments $T_k$ and the random variables $X^{(k)}_{T_k}$, in such a way to ensure the convergence of the weighted sums of random variables in $X_t$. 
{In \cite{AvedCos24game} we referred to the generalization of these results to different set ups as the  \emph{the game of mass}. 
In particular, \cite{AvedCos24game} abstracts} from the context of RWCRE by renaming increment sizes as a mass sequence $\mathbf{m}=(m_k,k \in \bb{N})$ and by considering $\tau(n) = \sum_{k =1}^n m_k$,  one is then lead to consider sums $X_t$ of the form \eqref{cooling-general-tau} and their averages $\mc{S}_t = X_t/t$ \eqref{gradual}, refereed to as gradual sums.
To deal with the non-centred case, by considering the sum $\mc{S}_t - \bb{E}[\mc{S}_t]$ we see that weak/strong convergence of $\mc{S}_t$ depends only on the convergence of $\bb{E}[\mc{S}_t]$.
This leads to the study of  structural conditions on the centering of the averages of the triangular sequence of random variables $\bb{E}[X^{(k)}_m]/m$ {and on the relative weights $\gamma_{k,n}$ of the increments considered up to time $t$,  recall \eqref{gamma-def}.
The guiding principle is that the considered mean observable (e.g. the displacement of RWCRE if one is interested in its asymptotic speed or the cumulant for the rate $n$ LDP) can be decomposed into independent block increments,
and the SLLN reduces to understanding how the \emph{mass} allocation across blocks weights these increments provided proper regularity conditions on the chosen mass sequence (i.e. on the corresponding resampling increments for the RWCRE model). 
To briefly discuss these regularities, let us first assume, as in \cite{AvedCos24game}, that}
\begin{equation}\label{mass:same}
\bb{E}[X^{(k)}_m]/m = \bb{E}[X^{(1)}_m]/m =: v_m \text{ for all }k \in \bb{N}, m >0,\\
\end{equation}
and that
\begin{equation}\label{mass:uniformity_limit}
 m\mapsto  v_m  \text{ is a continuous map such that } v_m\to v \text{ as } m \to \infty.
\end{equation}

Given a mass sequence $\mathbf{m} = (m_k,k \in \bb{N})$ with $m_k \in (0, \infty)$ set $\tau(k) = \tau^\mathbf{m} (k) := \sum_{i = 1}^{k}m_k$, let $\ell_t$ be as defined in \eqref{eq:ell_def} and let 
\begin{equation}\label{mass-n-def}
    m_{k,t}: = \tau(k) \wedge t - \tau(k-1) \wedge t.
\end{equation}
{Consider the 
the  \emph{empirical mass measure} denoted by $(\mu^{\mathbf{m}}_t, t \geq 0)$ and the \emph{empirical mass frequency} denoted by $(\mathsf{F}^{\mathbf{m}}_t, t \geq 0)$}, where

\begin{equation}\label{empiricalmass_empirical_freq}
\mu^{\mathbf{m}}_t 
: =  \sum_{k \in \bb{N}} \frac{m_{k,t}}{t} \delta_{m_{k,t}}, \quad \mathsf{F}^{\mathbf{m}}_t := \sum_{k \in \bb{N}}\frac{ \delta_{m_{k,t}}}{\ell_t}.
\end{equation}
{These two empirical measures allow to capture regularity structures in terms of the convergence of $(\mu^{\mathbf{m}}_t, t \geq 0)$ and $(\mathsf{F}^{\mathbf{m}}_t, t \geq 0)$. 
In particular, if $\mu^{\mathbf{m}}_t$ converges weakly to $\mu_*$ on the space of measures on $\bar{\bb{R}}$, the compactification of $\bb{R}$, see Definition 3.1 in \cite{AvedCos24game}, Proposition 3.1 in \cite{AvedCos24game} shows that $\bb{E}[\mc{S}_t]$ converges and the limit can be read out of $\mu_*$. Analogous conclusions can be obtained by working with the empirical mass frequency $(\mathsf{F}^{\mathbf{m}}_t, t \geq 0)$, though, the latter is subtler and we refer the reader to \cite{AvedCos24game} for a detailed discussion on these convergences and their relation. In particular, based on the proper notion of convergence of these empirical measures, and their relations, a rich variety of different explicit limits for $\bb{E}[\mc{S}_t]$ can be analysed, as elucidate by the many examples offered in \cite{AvedCos24game}.} 

\begin{remark}[{\bf Game of mass and interweaving of different processes}]
{The convergence of the empirical mass function mentioned above though sufficient is not a necessary condition for convergence of non-centred random variables. Indeed, if $\bb{E}[X^{(k)}_m]/m =1$ for all $k,m$ then \eqref{mass:uniformity_limit} holds true and $\bb{E}[\mc{S}_t] = \bb{E}[X_t/t] = 1$ for all mass sequence $\mathbf{m}$ and thus, for all resampling maps $\tau$ the LLN holds true according to the concentration properites of $\mathfrak{X}$.
One can relax assumption \eqref{mass:uniformity_limit} and allow for $\bb{E}[X^{(k)}_m]$ to also depend on $k$. When that is the case, the weak convergence of the empirical mass measure is no longer a sufficient condition to determine convergence of $\bb{E}[\mc{S}_t]$. Indeed, if $\bb{E}[X^{(k)}_m]/m = (-1)^k$ and $m_k = 2^{2^k}$ then $\mu_t^{\mathbf{m}} \to \delta_{\infty}$ and $\liminf \bb{E}[\mc{S}_t] = -1 \neq 1 = \limsup \bb{E}[\mc{S}_t]$. 
Yet, one can expand the notion of empirical mass measure to capture which increments are associated to which value $v$ and still characterize the limit provided weak convergence of this more general empirical measure. That is, if we let $v_{k,m}: = \bb{E}[X^{(k)}_m/m]$, the object to consider then is the family $(\mu^{\mathbf{m},v}_t 
, v,t \in \bb{R})$ where
\begin{equation}\label{empirical-v-decomposed}
\begin{aligned}
\mu^{\mathbf{m},v}_t 
&: =  \sum_{k \in \bb{N}} \frac{m_{k,t}}{t}\delta_{(-\infty, v_{m,k}]}, 
\end{aligned}
\end{equation}
and if $\mu^{\mathbf{m},v}_t \to \mu_*$, in a weak sense, then again in this case
$\lim_t\bb{E}[\mc{S}_t] = \int v \mu_*(dv)$. 
}
\end{remark}

\section{Fluctuations and general tools}\label{sec:Fluctuations}

The fluctuation theory for RWRE in one dimension,  originates in the recurrent Sinai--Kesten, \cite{S82, K86} picture and in the work of Kesten--Kozlov--Spitzer, \cite{KKS75_brief} which makes explicit that \emph{there is no single universal scaling} for RWRE fluctuations. In fact, the fluctuations are obtained from the study of hitting-times whose distribution depends on the tail of the ration of left to right jump probabilities $\rho = \frac{1-\omega}{\omega}$. 
When $\bb{E}[\log \rho]<0$, the correct normalisation and the resulting limit law depend on a parameter $s$ (\cite{KKS75_brief} write $\kappa$), defined implicitly by the condition $\bb{E}[\rho^{s}]=1$. 
We take the value  $s=0$ to represent the case  $\bb{E}[\log \rho] = 0$.  This choice splits the flutuation analysis of RWRE into six genuinely different regimes, that we organize into 3 regions and 3 (left) boundary cases according to the value of $s$ in table \ref{tab:classification}.

\begin{table}[h!]
    \centering
    \begin{tabular}{|c|c|}
    \hline
       Region  &   Boundary \\ \hline
        $s \in (0,1)$ sub-balistic & $s = 0$ Sinai regime\\
        $s \in (1,2)$ ballistic stable& $s = 1$ Cauchy regime\\
        $s \in (2, \infty)$ diffusive & $s = 2$ critically diffusive\\
        \hline
    \end{tabular}
    \caption{regime classification}
    \label{tab:classification}
\end{table}

Let $(Z_n, n \in \bb{Z}_+)$ be RWRE with environment law $\mu = \alpha^{\bb{Z}_+}$ that satisfies uniform ellipticity \eqref{eq:ue}, \eqref{non-lattice} and with fluctuation parameter $s \geq 0$. From \cite{K86, KKS75_brief},  it follows that there is a random variable $W_s$ and scaling factors $(a_n = a_{n,s},b_n = b_{n,s}, n \in \bb{N})$ for which
\begin{equation}\label{s-scaling-limit}
    \frac{Z_n - b_n }{a_n} \to W_s.
\end{equation}
Each $s\geq 0$ has its own scaling factors, centring prescriptions, and canonical limits, namely Sinai-Kesten law for $s = 0$, second-kind Mittag-Leffler laws for $s \in (0,1)$, stable laws for $s \in [1,2)$, Gaussian laws for $s\geq 2$, or critical variants with logarithmic scaling corrections for $ s = 1,2$). 

{It is important to note that, due to the techniques of renewal theory employed in the study of the fluctuations of the hitting times, the results presented in \cite{KKS75_brief,K86} rely on the assumption that \begin{equation}\label{non-lattice}
    \text{the law of $\log \frac{1-\omega}{\omega}$ is non-lattice},
\end{equation} 
that is the support of $\log \frac{1-\omega}{\omega}$, when
$\omega$ is sampled according to the environment law $\alpha$, is not contained in any affine lattice $a+h\mathbb Z$, with $a\in\mathbb R$ and $h>0$. 
the support of $Y$ is not contained in any affine lattice $a+h\mathbb Z$, with $a\in\mathbb R$ and $h>0$.}

At a heuristic level, varying $s$ changes the balance between localization, where a few atypical stretches of the environment control the fluctuations, and homogenised behaviour, where many contributions accumulate and Gaussian limits emerge; critical regimes interpolate between these mechanisms and may display additional boundary effects.

The purpose of this Section is to organise the RWCRE results against this RWRE benchmark. Under resampling the evolution turns into a concatenation of RWRE blocks, so that the overall fluctuations depend not only on the regime of $s$ but also on how the resampling map interacts with the corresponding scaling of block displacements (which is captured by the balance growth of the sums of $\tau(n)$- scaled  increments $T_1,\dots,T_n$). As a consequence, different regimes require different technical hypotheses (for instance, varying ellipticity and integrability assumptions) and, crucially, different notions of regularity for the cooling map tailored to the relevant scaling. 
{In Section \ref{ssec:fluctuation_overview} we state a single theorem that gathers all limiting laws currently available across the six regimes, thereby providing a unified reference point for the state of the art.}
In Section \ref{sec:S3} we explain the proof strategies and guiding ideas behind this theorem, emphasising how the same small set of mechanisms reappears in different guises: block decompositions and interweaving, criteria for dominance of the last block, replacement arguments, point-process descriptions in the stable regimes, the balance of rescaled pieces, and process-level convergence tools.

The results reviewed here were obtained over a sequence of works. The recurrent case and the classical diffusive regime (including $s=0$ and $s>2$) are treated in \cite{AveChidCosdHol22}. The regime $s\in(1,2)$, which is the most delicate after the critical case $s=1$, is analysed in \cite{AvedCosPet23}. The sub-ballistic regime $s\in(0,1)$ and the critical Gaussian regime $s=2$ are covered in \cite{dCosPetXie23b}.
Functional limit theorems { in RWCRE, when the $X^{(j)}$'s that make up $\mathfrak{X}$, see \eqref{underlying}, follow RWRE in the Sinai regime ($s =0$) and the cooling maps are polynomial ($T_k \sim \beta k^{\beta-1}$) or exponential ($T_k \sim Ck$),} 
are developed in \cite{Yon19}. 
{In Section \ref{subsec:recfl}} we state a single theorem that summarises the known fluctuation limits across these regimes (with $s=1$ the only case not addressed in the present state of the art), and we indicate how the scaling and the limiting laws change as the parameter $s$ varies.


\subsection{Framework and Notation}
\label{ssec:fluctuation_overview}

The goal of fluctuation theory is to describe the scaling limits of the position
$X_n$ as $n\to\infty$, that is, the weak limits of suitably normalised and centred
random variables of the form
\[
\mc{X}_n: =\frac{X_n-b_n}{a_n},
\]
for deterministic sequences $(a_n)_{n\ge1}$ with $a_n>0$ and $(b_n)_{n\ge1}$.
When a full limit does not exist, one seeks instead a complete description of
the set of subsequential limits,  limit points, of $\big((X_n-b_n)/a_n\big)_{n\ge1}$.
To determine the normalisation  and centering $(a_n, b_n)_{n \geq 1}$ is already a
delicate matter for RWRE and, in the cooling setting, becomes substantially more
flexible: for each fixed RWRE regime,encoded by the parameter $s$, the choice of
$(a_n,b_n)_{n \geq 1}$ depends on the resampling map $\tau$, and may
interpolate between the static RWRE behaviour, when $\tau$ is such that few large blocks dominate, and
a homogenised behaviour, when $\tau$ is such that many blocks contributing on the relevant scale.

As a first step we we take the centering term $b_n:=\bb{E}[X_n]$,
and we focus on identifying a suitable normalising sequence $(a_n)_{n\in\bb{N}}$.
This choice is essentially unique, in the sense that replacing $(a_n)$ by another
normalisation can only rescale subsequential limits by a deterministic constant.

More precisely, let $(a_n)_{n\in\bb{N}}$ and $(a_n')_{n\in\bb{N}}$ be two sequences with
$a_n>0$ and $a_n'>0$, and define the centred-and-scaled variables
\[
\mc{X}_n:=\frac{X_n-\bb{E}[X_n]}{a_n},
\qquad
\mc{X}_n':=\frac{X_n-\bb{E}[X_n]}{a_n'} \;=\; \frac{a_n}{a_n'}\,\mc{X}_n.
\]
Suppose there exists a subsequence $(n_k)_{k\ge1}$ such that
\[
\mc{X}_{n_k}\ \Rightarrow\ \mc{X}_*,
\qquad\text{with $\mc{X}_*$ non-degenerate,}
\]
and assume that along the same subsequence the ratio of scales converges,
\[
\frac{a_{n_k}}{a_{n_k}'} \longrightarrow A_* \in [0,\infty].
\]
Then
\[
\mc{X}_{n_k}' \ =\ \frac{a_{n_k}}{a_{n_k}'}\,\mc{X}_{n_k}
\ \Rightarrow\ A_*\,\mc{X}_*.
\]
In particular, whenever $A_*\in(0,\infty)$ the two normalisations lead to the
same family of limit points up to deterministic multiplication. If $A_*  = 0$ then the limit distribution of $(\mc{X}_{n_k}', k \in \bb{N})$ is $\delta_0 $ which is the  distribution of a degenerate random variable. If $A_*  = \infty$ then the limit distribution of $(\mc{X}_{n_k}', k \in \bb{N})$ is not a probability distribution. Thus, once one
has identified a scale $a_n = a_n(\mathfrak{X},\tau) = a_n(s,\tau)$, as a function of the cooling map $\tau$ and the RWRE pieces or, better yet, its
parameter $s$, the remaining task is to describe the
resulting subsequential limits and to determine when the centering $\bb{E}[X_n]$
may be replaced by a more explicit approximation (such as a linear drift) without
changing the limiting laws.

\noindent{\bf Fluctuation decomposition and mass balance.}
Let $\mathfrak{X}$ be a family of underlying processes as in \eqref{underlying} and consider a resampling map \eqref{basic-resampling-tau}.  It follows from \eqref{cooling-general-tau} that 
\begin{equation}\label{fluctuation-basic-decomposition}
\begin{aligned}
    \mc{X}_n = \mc{X}_n(\mathfrak{X}, \tau) &:=\frac{X_n - \bb{E}[X_n]}{a_n} = \sum_{k} \frac{X^{(k)}_{T_{k,n}} - \bb{E}[X^{(k)}_{T_{k,n}}]}{a_n},
\end{aligned}
\end{equation}
where $ T_{k,n,\tau} = T_{k,n} := (\tau(k)\wedge n)-(\tau(k-1) \wedge n)$, recall \eqref{gamma-def}.
Once we consider the basic decomposition of $\mc{X}_n$ into independent pieces as in \eqref{fluctuation-basic-decomposition} we can shift our attention to pieces and their limiting behaviours. From classical RWRE theory, see \cite{S82, KKS75_brief,Zei04}, it follows that for each $s\geq 0$ there is $\alpha_T =\alpha(s,T)$ such that $(X^{(k)}_T - \bb{E}[X^{(k)}_T])/\alpha_T \to W_s$ 
where $W_s$ denotes the corresponding standard RWRE fluctuation limit (Sinai--Kesten,
Mittag--Leffler, stable, or Gaussian, depending on $s$)

{Motivated by this, we rewrite \eqref{fluctuation-basic-decomposition} as $\mc{X}_n
=\sum_{k\geq 1}\lambda_{k,n}\,\mc{Y}_{k,T_{k,n}}$, where
\begin{equation}\label{fluctuation-mass-balance}
\mc{Y}_{k,T}:=\frac{X^{(k)}_T-\bb{E}[X^{(k)}_T]}{\alpha(k,T, \tau)},
\;
\lambda_{k,n,\tau}=\lambda_{k,n}:=\frac{\alpha(k,T, \tau)}{a_n}.
\end{equation}
The coefficients $\lambda_{n,\tau} =\lambda_n = (\lambda_{k,n})_{k\ge1}$ form the \emph{fluctuation mass profile} at time $n$. 
The vector $\lambda_n$ quantifies how much each block contributes on the global scale $a_n$. In the following sections we shall see that for $s \in (1,2] $ the scaling parameter $\alpha( k,T, \tau)$  depends on the map $\tau$. On the other hand for $s \in \bb{R} \setminus [1,2]$ the scaling parameter does not depend on $\tau$. Indeed, for $s \in \bb{R} \setminus [1,2]$ one can take the diffusive scaling $\alpha( k,T, \tau) =\alpha_T = c_T\sqrt{T}$ with $\lim_{T} c_T = c_s \in (0,\infty)$. 
Determining the correct normalisation $a_n$ is therefore equivalent to identifying $\alpha( k,T, \tau)$, that is, identifying
the scale on which the profile $\lambda_{k,n}$ has a non-trivial mass balance,
i.e.\ neither collapses to a single block nor spreads too thinly for all blocks to vanish.
This is precisely where the cooling map enters in an essential way: through the array
$(T_{k,n})_{k\ge1}$, the map $\tau$ shapes the profile $\lambda_{k,n}$ and hence
controls both the correct scaling and the possible limiting laws.}

\subsection{Scaling limit via variance scaling for the recurrent, sub-ballistic and diffusive regimes:
\texorpdfstring{$s \in [0,1)\cup(2,\infty)$}{s in [0,1) or in (2,infinity)}}\label{Ssec:Replacement_works}

\label{subsec:recfl}

In this section we collect the  statements that enable us to identify the scaling limits of RWCRE for the recurrent regime ($s = 0$), for the sub-ballistic transient regime($s \in (0,1)$), and for the diffusive regime ($s>2$)  that is, the cases when the underlying processes $X^{(k)}$ that make up $\mathfrak{X}$ in \eqref{underlying} is distributed as a RWRE that satisfies $\bb{E}[\log \rho] = 0$ or $\bb{E}[\log \rho] <0$ and $\bb{E}[\rho^s] = 1$ with $s \in (0,1) \cup (2, \infty)$. 

In \cite{AveChidCosdHol22} and \cite{dCosPetXie23b} it is shown that, when $s \in [0,1) \cup (2, \infty)$, the convergence of RWRE occurs in $L^2$ and therefore the scaling of each piece simplifies and can be taken as a function of the size $T_{k,n}$. 
In this case the variance, $\var(X) := \bb{E}[(X - \bb{E}[X])^2]$, is the key to determine the scaling limits: we set  
\begin{equation}\label{alpha-a-var-scaling}
\alpha(k,T,\tau) := \sqrt{\var(X^{(k)}_T)} \quad \text{ and } \quad a_n := \sqrt{\var(X_T)}.
\end{equation}
Theorem~\ref{thm:rlpts} below gives a characterisation of the possible limit points (as mixtures of the limits laws $W_s$ with a Gaussian random variable $Z$.

From the choice of $\alpha(k,T,\tau)$, $a_n$, and \eqref{fluctuation-mass-balance} we obtain that 
\begin{equation}\label{variance-lambda-scaling}
\lambda_{\tau,n}(k) = \sqrt{\var(X^{(k)}_T)}/\sqrt{\var(X_T)}.
\end{equation}
By the independence of the components of $\mathfrak{X}$, see \eqref{eq:X_dec}, it follows that $\lambda_{\tau,n}$ is a vector with unit $\ell_2(\bb{N})$-norm, i.e., $\norm{\lambda_{\tau,n}}_{2}^2 := \sum_{k\in\bb{N}} \lambda_{\tau,n}(k)^2 = 1$. 
Let $(W_j)_{j \in \bb{N}}$ be  a family of i.i.d.\ random variables with same law as $W_s$, a random variable with finite variance when $s \in [0,1) \cup (2,\infty)$, and let $\sigma_s: = \sqrt{\var(W_s)}$. 
Define for $\lambda = \prt{\lambda(j)}_{j \in \bb{N}} \in \ell_2(\bb{N})$,  the $\lambda$-mixture of normalised $W$ random variables by
\begin{equation} \label{e:Smix} 
    W^{\otimes \lambda} := \sum_{j\in\bb{N}} \lambda(j) (\sigma_s^{-1}W_j)= \lim_{n \to \infty} \sum_{j=1}^n\lambda(j) (\sigma_s^{-1}W_j),
\end{equation}
where the above (almost sure) limit  is well defined  from the convergence in $L^2$ of the partial sums.
We present in a single framework the content of Theorem 2 in \cite{AveChidCosdHol22}, which gives the scaling limits of RWCRE in the case $s = 0$ and $s>2$, and Theorem 1.3 in \cite{dCosPetXie23b}, which gives the scaling limits of RWCRE in the case $s \in(0,1)$. 

For $\lambda \in \ell_2(\bb{N})$, let $\lambda^\downarrow$ be the vector obtained from $\lambda$ by reordering the entries of $\lambda$ in decreasing order, see . Consider the equivalence relation $\lambda\sim \lambda'$ when $\lambda^\downarrow =\lambda'^\downarrow$ and put $[\lambda] := \chv{\lambda' \in \ell_2(\bb{N})\colon \lambda' \sim \lambda }$. 
In what follows, $(n_i)_{i\in\bb{N}}$ denotes a strictly increasing sequence of integers.

\begin{theorem}[{\bf Limit distributions for RWCRE for $s \in[0,1)\cup (2,\infty)$}] \label{thm:rlpts}
Assume $\mu = \alpha^{\bb{Z}}$ satisfies \eqref{eq:mu_iid}, with $\alpha$ a uniformly elliptic probability measure on $(0,1)$, \eqref{eq:ue}, non-lattice, \eqref{non-lattice}, and fluctuation parameter $s \in [0,1) \cup (2,\infty)$. Assume that $(X^{(k)}, k \in \bb{N})$ are independent and that $ X^{(k)}= (X^{(k)}_n, n \in \bb{Z}_+)$ is distributed according to\eqref{eq:annealed_rwre}, $ X^{(k)}$ is RWRE under the annealed law $\mu$ fixed above. Let $\tau$ be a resampling map, as in  \eqref{cooling-general-tau} and let $ \mathfrak{X}=\prt{\mathfrak{X}_n, n\in\bb{Z}_+}$ be the sequence with $\mathfrak{X}_n$ as defined in \eqref{fluctuation-basic-decomposition} when $a_n$ is given by \eqref{alpha-a-var-scaling}. Then the sequence $\mathfrak{X}$ is tight in the weak topology and its limit points are characterised as follows. Let $\lambda_{\tau,n} = (\lambda_{\tau,n}(k), k \in \bb{N})$ with $\lambda_{\tau,n}(k)$ as defined in \eqref{variance-lambda-scaling}. If $(n_i)_{i\in\bb{N}_0}$ is such that
\begin{equation}\label{diagc}
\lim_{i\to\infty} \lambda^{\downarrow}_{\tau,n_i}(k) =: \lambda_*(k) \qquad \forall\,k \in \bb{N}_0,
\end{equation}
then $\lambda_* = (\lambda_*(k), k \in \bb{N}) \in \ell_2(\bb{N})$ and
\begin{equation}\label{lpoints0}
\mathfrak{X}_{n_i}
\todtwo
W_s^{\otimes \lambda_*}
+ a(\lambda_*)\,\Phi \qquad \forall\,p>0,
\end{equation}
where $a(\lambda_*) := (1 - \|\lambda_*\|_2^2)^{\tfrac12}$, $\Phi$ is a standard normal random variable, and $W_s^{\otimes \lambda_*} $ is as in \eqref{e:Smix}.
\end{theorem}
\begin{proof}
    The proof of this results follows from the replacement method explained in Section 1.2.1 of \cite{dCosPetXie23b}, see in particular Theorem 1.7 of \cite{dCosPetXie23b}. An important step that is required in each case is the control of the $L^2$ convergence of RWRE which is proved in the case $s = 0$ in Appendix C of \cite{AdH17springer}, in the case $s \in (0,1)$ in Theorem 1.10 in \cite{dCosPetXie23b}, the case $s>2$ is covered in Theorem 3 of  \cite{AveChidCosdHol22}.
\end{proof}
\begin{remark}\label{R:fatou}\rm{}
We note that if one is allowed to take subsequences, then condition \eqref{diagc} is not restrictive. Indeed, since for any $k \in \bb{N}_0$ and $n \in \bb{N}_0$ the value $\lambda^{\downarrow}_{\tau,n}(k)$ belongs to $[0,1]$, if we take a (diagonal) subsequence $(n_i)_{i \in \bb{N}}$, then condition \eqref{diagc} will be satisfied for some vector $\lambda_*$. By Fatou's lemma:
  \begin{equation}\label{Fatousub}
    \sum_{k \in \bb{N}} \lambda^2_*(k) =    \sum_{k \in \bb{N}} \liminf_{i \in \bb{N}}\prt{\lambda^{\downarrow}_{\tau,n_i}(k)}^2 \leq \liminf_{i \in \bb{N}} \sum_{k \in \bb{N}} \prt{\lambda^{\downarrow}_{\tau,n_i}(k)}^2 = 1,
  \end{equation}
  which guarantees that $\lambda_* \in \ell_2(\bb{N})$ and so $W_s^{\otimes \lambda_*}$ in  \eqref{lpoints0} is well defined. With this it follows that Theorem~\ref{thm:rlpts} characterizes \emph{all limit points} of  $\prt{\mathfrak{X}_n, n\in\bb{Z}_+}$.  
\end{remark}

\begin{remark}
    It is possible to capture the transition between homogeneous RWRE and static RWCRE in the fluctuations by exploring different choices of $\tau(n)$ with different growth patters. In such cases we can describe how the scaling factors $a_n,b_n$ and scaling limits $W^\otimes \lambda_*$ depend on $\tau$ through a calculation of the value of $\var (X_n)$ in terms of $n$ and $\tau$. In \cite{AveChidCosdHol22}, the authors have shown, for the case $s = 0$, in Examples (Ex.3) -- (Ex.6) the transition in $(a_n,b_n)$ and limit points that occur when we move from ``polynomial'', with $T_k \sim B k^\beta$, to ``exponential'', with $T_k \sim \exp(ck)$, ``double exponential'', with $T_k \sim \exp (\exp (ck))$, and to ``triple exponential'', with $T_k \sim \exp(\exp(\exp(ck)))$, cooling maps.    

    Also interesting is the exploration of possible limit points one can obtain by changing $\tau$, in Example 5.5 of \cite{dCosPetXie23b} it is shown, for the case $s \in (0,1)$, how to obtain arbitrary mixtures of $W_s$ and Gaussian random variables for subsequences of $\mathfrak{X}_n$. 

    When $s>2$ the family $\mathfrak{X}_n$ admits a unique limit point, namely the  Gaussian random variable. This is due to the fact that for $s>2$ $W_s$ is a Gaussian randon variables and variance $1$ mixtures of independent copies of $W_s$ have the standard Gaussian distribution. However, we point out that the scaling factors $(a_n,b_n)$ may depend of the choice of $\tau$. In order to obtain convergence of the sequence with $a_n \sim \sqrt{n}$ it is required that $\var(X_{n_i})/n_i$ converges along the sequence $(n_i, i \in \bb{N})$. When $\tau$ is a cooling map, satisfying \eqref{effective-cooling}, then $a_n\sim \sigma_s \sqrt{n}$ with $\sigma^2_s = \lim_n \var(Z_n)/n$. Though the diffusive scaling, $a_n \sim \sqrt{n}$, holds for any map for which  $\var(X_{n_i})/n_i$ converges it is worth pointing out that the centering term $(b_n)$ can not in general be determined, see example (Ex.8) in \cite{AveChidCosdHol22}. However, if the limit vector $\lambda_*$ determined by \eqref{diagc} satisfies 
    \[
    \sup_n \sum_{k} \lambda_{\tau,n}(k) <\infty
    \]
    then one may take $b_n = v_\mu n$ where $v_\mu  = \lim_n \bb{E}[Z_n]/n$ is the speed of the underlying RWRE.
\end{remark}

\subsection{Fluctuations when the replacement method fails, the ballistic and critically diffusive regimes:
\texorpdfstring{$s \in (1,2]$}{s in (1,2]}}
\label{Ssec:replacement_fails}
\label{sec:S3}
\subsubsection{The critical Gaussian regime:
\texorpdfstring{$s = 2$}{s = 2}}

The key to the regime $s=2$ is the control of the second moment. In this case, the underlying RWRE still has a Gaussian scaling limit, but it sits at the borderline where the second moment already feels the contribution of moderate slowdown probabilities. More precisely, if $Z$ is a $2$-regular RWRE, then there exists a constant $b>0$ such that
\begin{equation}\label{rwre-s2-gaussian-limit-survey}
\frac{Z_n-nv}{b\sqrt{n\log n}}\tod \mc{N},
\end{equation}
where $\mc{N}$ is a standard Gaussian random variable. At the same time, there exists a constant $K_0>0$ such that the one-sided moderate slowdown asymptotics
\begin{equation}\label{rwre-s2-moderate-slowdown-survey}
P^\mu_0(Z_n-nv<-x)\sim K_0 (nv-x)x^{-2},
\qquad \sqrt{n\log^3 n}\le x\le nv-\log n,
\end{equation}
hold uniformly in the indicated range.

These slowdown estimates generate an additional contribution to the second moment, yielding
\begin{equation}\label{rwre-s2-second-moment-survey}
\lim_{n\to\infty}E^\mu_0\!\left[\left(\frac{Z_n-nv}{\sqrt{n\log n}}\right)^2\right] = b^2+K_0v.
\end{equation}
In particular,
\begin{equation}\label{rwre-s2-mean-var-survey}
\frac{E^\mu_0[Z_n]-nv}{\sqrt{n\log n}}\to 0,
\qquad
\frac{\Var(Z_n)}{n\log n}\to b^2+K_0v.
\end{equation}
Hence the convergence in \eqref{rwre-s2-gaussian-limit-survey} cannot be improved to $L^2$-convergence: the variance asymptotic is strictly larger than the square of the scaling constant appearing in the limiting theorem. Equivalently, one may rewrite the Gaussian limit as
\begin{equation}\label{rwre-s2-variance-restatement-survey}
\frac{Z_n-E^\mu_0[Z_n]}{\beta\sqrt{\Var(Z_n)}}\tod \mc{N},
\qquad
\beta=\frac{b}{\sqrt{b^2+K_0v}}<1.
\end{equation}

This mismatch between the RWRE scaling and the variance is the basic phenomenon behind the critically diffusive RWCRE regime: already at the RWRE level, Gaussian fluctuations persist, but they are not captured by the full standard deviation.

\subsubsection{The ballistic regime: \texorpdfstring{$s \in (1,2)$}{s in (1,2)}}
The replacement method employed in Theorem \ref{thm:rlpts} depends on the convergence in $L^2$ of RWRE. When the fluctuation parameter $s \in [1,2]$, convergence in $L^2$ does not hold. Indeed, for $s \in [1,2)$ the scaling limit $W_s$, as established in \cite{KKS75_brief}, is a stable law that does not admit a second moment. In this regime one can at best prove convergence of the RWRE in $L^p$ for $p<s$; see also \cite{AvedCosPet23}.

When $s \in (1,2)$, the scaling of convergence of RWRE is of order $n^{1/s}$, that is,
\[
n^{-1/s}(Z_n-\bb{E}[Z_n]) \tod W_s,
\]
where $W_s$ is the stable random variable (mean zero, totally skewed to the left, with scale parameter $b>0$) with characteristic function
\[
\bb{E}[\exp(iu W_s)] 
= 
\exp\bigg[ -b|u|^s\bigg(1 + i \frac{u}{|u|} \tan \Big(\frac{s \pi}{2}\Big)\bigg)\bigg], 
\qquad u \in \bb{R}.
\]
On the other hand, as shown in Theorem 3.8 of \cite{AvedCosPet23},
\[
\var(Z_n)\asymp n^{3-s},
\]
so that the standard deviation is of order $n^{(3-s)/2}$, which is larger than $n^{1/s}$. Thus, in the range $s\in(1,2)$, there are two natural candidate scales for fluctuations of RWCRE, namely the variance scale and the stable scale.

When the cooling map is slow enough, as captured by the condition
\begin{equation}\label{slow-cooling-s-stable}
\lim_n \sup_{k \leq n} \frac{T_k}{\big(\sum_{k = 0}^n T_k^{3-s}\big)^{1/2}} = 0,
\end{equation}
homogenization occurs for the fluctuations of $\mathfrak{X}$, and the choice of the variance scaling yields
\begin{equation}\label{normal-limit-s-stable-slow-cooling}
    \frac{X_n - \bb{E}[X_n]}{\sqrt{\var(X_n)}} \tod \mc{N},
\end{equation}
where $\mc{N}$ is a standard Gaussian distribution. 
The interpretation of \eqref{slow-cooling-s-stable} is that no single block is large compared to the total Gaussian fluctuation scale; hence many blocks contribute, each one weakly. This is the structural reason for the Gaussian limit; see Theorem 3.2 of \cite{AvedCosPet23}.

For the complementary regime, it is natural to compare the block lengths $T_k$ with the stable scale $\tau(n)^{1/s}$. To this end, one considers cooling maps for which the asymptotic profile exists:
\begin{equation}\label{g-profile-s-stable}
g(x):=\lim_{n\to\infty}\frac{\sum_{k=1}^n T_k \mathbf{1}_{\{T_k<x\,\tau(n)^{1/s}\}}}{\tau(n)},
\qquad x>0.
\end{equation}
Thus $g(x)$ records the asymptotic fraction of the total time $\tau(n)$ spent in blocks whose lengths are smaller than $x\,\tau(n)^{1/s}$, that is, in blocks which are subcritical on the $s$-stable scale. In particular,
$
g(\infty):=\lim_{x\to\infty}g(x)\in[0,1]
$
measures the total proportion of time carried by blocks which are at most of order $\tau(n)^{1/s}$. Hence the condition $g(\infty)=0$ expresses a fast-cooling regime: asymptotically, only blocks much larger than the stable threshold $\tau(n)^{1/s}$ contribute to the total running time.

Associated with the profile $g$ is the Lévy density $\lambda_g:(-\infty,0)\to[0,\infty)$ given by
\begin{equation}\label{lambda-g-definition-survey}
\lambda_g(-t)
=
K_0 t^{-s}\int_{t/\nu_\mu}^{\infty}
\left(\frac{\nu_\mu^s}{t}-\frac{s-1}{x}\right)\,g(dx),
\qquad t>0,
\end{equation}
and we denote by $W_{\lambda}$ the infinitely divisible random variable with characteristic function
\begin{equation}\label{W-lambda-charfun}
\bb{E}\big[e^{itW_{\lambda}}\big]
=
\exp\left(
\int_{-\infty}^0
\big(e^{itx}-1-itx\big)\lambda(x)\,dx
\right),
\qquad t\in\bb{R}.
\end{equation}
The law $W_{\lambda_g}$, referred to as a tempered stable law, describes the infinitely divisible contribution generated by the subcritical part of the cooling profile.

Provided \eqref{g-profile-s-stable} holds, one can distinguish two sufficient conditions for stable-type limits:
\begin{itemize}
 \item[{\bf(S1)}] 
  $\sup_n \sum_{k=1}^n \frac{T_k^{1/s}}{\tau(n)^{1/s}}< \infty$,
  and 
  $\underset{n\to \infty}\lim \sum_{k=1}^n \frac{T_k^{1/s}}{\tau(n)^{1/s}} \ind{T_k<m}=0$ for all $m<\infty$.
 \item[{\bf(S2)}] 
  $\lim_{n\to\infty} \frac{\max_{k\leq n} T_k(\log T_k)^{4s}}{\tau(n)} = 0$.
\end{itemize}

Under condition {\bf(S1)}, the cooling map is fast enough so that the subcritical contribution is asymptotically negligible, and one recovers the pure RWRE stable law
\begin{equation}\label{stable-limit-S1-survey}
\frac{X_n-\bb{E}[X_n]}{n^{1/s}}\tod W_s.
\end{equation}
Under condition {\bf(S2)}, the limit depends on the full cooling profile $g$, and one has
\begin{equation}\label{stable-limit-S2-survey}
\frac{X_n-\bb{E}[X_n]}{n^{1/s}}
\tod
\begin{cases}
W_s, & \text{if } g(\infty)=0,\\[1mm]
W_{\lambda_g}+(1-g(\infty))^{1/s}W_s, & \text{if } g(\infty)\in(0,1],
\end{cases}
\end{equation}
where $W_s$ is independent of $W_{\lambda_g}$. Thus, when $g(\infty)\in(0,1]$, the subcritical part of the cooling profile contributes through the tempered stable component $W_{\lambda_g}$, while the remaining supercritical contribution produces the independent stable term $(1-g(\infty))^{1/s}W_s$.

The combination of the Gaussian regime \eqref{normal-limit-s-stable-slow-cooling} with the stable-type limits \eqref{stable-limit-S1-survey}--\eqref{stable-limit-S2-survey} shows that, in the ballistic and critically diffusive range $s\in(1,2)$, the fluctuation behavior of RWCRE is governed by a competition between the Gaussian variance scale and the stable scale $n^{1/s}$.

From Theorem 3.1 in \cite{AvedCosPet23}, and as a consequence of the above discussion we can make the crossover explicit for polynomial cooling maps.
For polynomial cooling maps, the above dichotomy becomes a sharp trichotomy. More precisely, assume that the cooling increments have eventual polynomial growth,
\begin{equation}\label{polygrow}
  \lim_{k\to\infty} \frac{T_k}{A k^a} = 1,
  \qquad \text{for some } A,a \in (0,\infty).
\end{equation}
Then the fluctuation behavior of RWCRE undergoes a phase transition at the critical exponent $a_c=(s-1)^{-1}$. If $a<(s-1)^{-1}$, the cooling is slow enough for homogenization to prevail, and the fluctuations are asymptotically Gaussian:
\begin{equation}\label{gauss_slow}
  \frac{X_n - \bb{E}[X_n]}{Bn^{\beta} }
  \tod \mc{N},
\end{equation}
where
\begin{equation}\label{beta-poly-survey}
\beta := \frac{a(3-s)+1}{2(a+1)},
\end{equation}
and $B$ can be computed explicitly, through a somewhat complicated formula, see Theorem 3.1 in \cite{AvedCosPet23} for details.

At the critical value $a=(s-1)^{-1}$,  one obtains the generalized tempered stable limit
\begin{equation}\label{Xnstableliml}
  \frac{X_n - \bb{E}[X_n]}{n^{1/s}}
  \tod W_{\lambda_{c,r}},
\end{equation}
where $W_{\lambda_{c,r}}$ is defined by \eqref{W-lambda-charfun} with
\begin{equation}\label{lcr}
 \lambda_{c,r}(x) =  c|x|^{-s-1}(1+x/r)_+, \qquad x<0,
\end{equation}
for some $c,r \geq 0$.

Finally, if $a>(s-1)^{-1}$, the cooling is fast enough for the static RWRE stable mechanism to dominate, and one recovers the pure stable limit
\begin{equation}\label{Xnstablelim}
    \frac{X_n - \bb{E}[X_n]}{n^{1/s}}
    \tod W_s.
\end{equation}
Thus polynomial cooling makes explicit the crossover from Gaussian to stable fluctuations, with a critical generalized tempered stable law appearing exactly at the transition.

\appendix
\renewcommand\thesection{\Alph{section}}
\renewcommand\theequation{\Alph{section}.\arabic{equation}}
\setcounter{equation}{0}
\section{Proof of LDP theorem}
\label{app:LDP}
The first step in the proof of the annealed LDP is the derivation of the scaled cummulant generating function (s.c.g.f.). First we note that for $Z$, the static RWRE, the s.c.g.f. converges:
\begin{equation}\label{scgf-static}
\lim_T \frac{1}{T} \log \bb{E}[e^{\theta Z_T}] = \Lambda(\theta) \quad \text{ for all } \theta \in \bb{R}.
\end{equation}
Let $I_a: \bb{R} \to [0, \infty]$ be the rate function for RWRE $Z$ under the annealed law, whose formula can be found in \cite{ComGanZei99}.

For RWCRE $X$ the s.c.g.f. can be derived by an application of Toeplitz Lemma \cite[Lemma 1]{AdH17springer} together with \eqref{scgf-static}, that is, for all $\theta \in \bb{R}$,
\begin{equation}\label{scgf-cooling}
    \begin{aligned}
        \lim_n \frac{1}{n} \log \bb{E}[e^{\theta X_n}]&= \lim_n \sum_k \frac{T_{k,n}}{n} \frac{1}{T_{k,n}} \log \bb{E}[e^{\theta Z_{T_{k,n}}}] = \Lambda(\theta).
    \end{aligned}
\end{equation}
Equation \eqref{scgf-cooling} indicates that the RWCRE process $X$ satisfies the LDP  with same speed and rate function as the RWRE process $Z$. Indeed, if we knew that $X$ satisfies the LDP with a convex rate function, then by taking the Legendre transform of $\Lambda$ in \eqref{scgf-cooling} we would be able to identify the rate function. The remaining challenge is to prove that $X$ satisfies the LDP with a convex rate function.
In order to prove the annealed LDP for $X$ we will show that 
for every open set $O \subset \bb{R} $
\begin{equation}\label{Open-LDP}
    \liminf_{n \to \infty} \frac{1}{n} \log \bb{P}\bigg(\frac{X_n}{n} \in O\bigg) \geq - \inf_{x \in O} I_a(x),
\end{equation}
and that for every closed set $C \subset \bb{R}$
\begin{equation}\label{closed-LDP}
    \limsup_{n \to \infty} \frac{1}{n} \log \bb{P}\bigg(\frac{X_n}{n} \in C\bigg) \leq - \inf_{x \in C} I_a(x),
\end{equation}

Let $[[a]] := [a-1,a+1]$. Note that since $X_n \in \bb{Z}$ is a nearest neighbour random walk starting from $0$, there is $[[a]]_n \in [a-1,a+1]$ for which $\bb{P}(X_n = [[a]]_n \mid X_n \in [[a]]) = 1$ and thus, in order to obtain \eqref{Open-LDP} and \eqref{closed-LDP} we shall first prove that for every $x \in [-1,1]$
\begin{equation}\label{point-rate-function-notation}
\begin{aligned}
    I_a(x) &\leq \liminf - \frac{1}{n} \log \bb{P}(X_n  = [[nx]]_n) \\
    &\leq \limsup - \frac{1}{n} \log \bb{P}(X_n = [[nx]]_n) \leq I_a(x).
    \end{aligned}
\end{equation}

\paragraph{Proof of the upper bound in \eqref{point-rate-function-notation}.}
Since 
\[
\{X_n = [[nx]]_n\} \supset \{X_n = [[nx]]_n, X_{\tau(i)} = [[\tau(i)x]]_{\tau(i)}, i \leq \ell_n\},
\]
it follows that
\begin{equation}
\begin{aligned}
-\frac{1}{n} \log \bb{P}_0(X_n = [[nx]]_n) 
&\leq -\frac{1}{n} \log \bb{P}(Z_{T_{i,n}} = [[T_{i,n}x]]_{T_{i,n}}, i >0 ) \\
&\leq \sum_k \gamma_{k,n} \bigg[-\frac{1}{T_{k,n}}\log \bb{P}(Z_{T_{i,n}} = [[T_{i,n}x]]_{T_{i,n}})\bigg] \\
& \to I_a(x),
\end{aligned}
\end{equation}
where the last passage follows from Toepliz lemma, see \cite[Thm. 1.2.3, p.36]{Rev68} and \cite[Lemma 1]{AdH17springer},  together with the fact that $\lim_T \frac{1}{T} \log \bb{P}(Z_n = [[nx]]_n) = -I_a(x)$.
 \paragraph{Proof of the lower bound in \eqref{point-rate-function-notation}.}
The lower bound follows from an exponential Markov inequality
\begin{equation}
    \begin{aligned}
        \bb{P}(X_n = [[nx]]_n) \leq \min\{\bb{P}(X_n \geq [[nx]]_n), \bb{P}(X_n \leq [[nx]]_n)\}
    \end{aligned}
\end{equation}
Therefore  
\begin{equation}
    \begin{aligned}
\bb{P}(X_n = [[nx]]_n) &\leq      \bb{P}(X_n \geq [[nx]]_n) \leq \inf_{\theta>0}\frac{1}{e^{\theta [[nx]]_n}}\bb{E}(e^{\theta X_n}), \\
\bb{P}(X_n = [[nx]]_n)&\leq \bb{P}(X_n \leq [[nx]]_n) \leq \inf_{\theta<0}\frac{1}{e^{\theta [[nx]]_n}}\bb{E}(e^{\theta X_n}).
    \end{aligned}
\end{equation}
Now we note that $\lim_n \frac{1}{n} \log \bb{E}[e^{\theta X_n}] = \Lambda(\theta) =   I^*_a(\theta)$.
To conclude note that for $x>0$
\[
\liminf-\frac{1}{n} \log \bb{P}(X_n  \geq [[nx]]_n) \geq \sup_{\theta>0} \big( \theta x - I^*_a(\theta)) = I_a(x),
\]
and for $x<0$
\[
\liminf - \frac{1}{n} \log \bb{P}(X_n  \leq [[nx]]_n) \geq \sup_{\theta<0} \big( \theta x - I^*_a(\theta)) = I_a(x).
\]

\paragraph{Proof of LDP given \eqref{point-rate-function-notation}}In what follows we prove that \eqref{point-rate-function-notation} implies \eqref{Open-LDP} and \eqref{closed-LDP}.

To see \eqref{Open-LDP}, fix $O\subset \bb{R}$ an open set and  note that for every $x \in O$ there is $n_0$ such that for $n>n_0$ $\frac{[[nx]]_n}{n} \in O$. Now 
\begin{equation} \label{liminf-open-ldp}
\begin{aligned}
\liminf_{n\to\infty}\frac{1}{n}\log \bb{P}\left(\frac{X_n}{n}\in \mathcal O\right)
&= -\limsup_{n\to\infty}\left(-\frac{1}{n}\log \bb{P}\left(\frac{X_n}{n}\in \mathcal O\right)\right)\\
&\ge \inf_{x \in O}-\limsup_{n\to\infty}\left(-\frac{1}{n}\log \bb{P}\left(X_n=[[ nx]]_n\right)\right)\\
&= \inf_{x \in O}-I_a(x). 
\end{aligned}
\end{equation}

For \eqref{closed-LDP}, note that since $\bb{P}(X_n  \in[-n,n]\cap \bb{Z}) = 1$ it follows  that 
\[
\bb{P}\left(\frac{X_n}{n}\in C\right)
\le (2n+1)\,\sup_{x\in C}\bb{P}\left(X_n=[[nx]]_n\right),
\]
and thus
\begin{equation}
\label{proto-n-integer}    
\limsup_{n\to\infty}\frac{1}{n}\log \bb{P}_0\left(\frac{X_n}{n}\in C\right)
=
\limsup_{n\to\infty}\ \sup_{x\in C}\ \frac{1}{n}\log \bb{P}_0\left(X_n=[[nx]]_n\right).
\end{equation}

To prove \eqref{proto-n-integer},  it suffices to show that for every $\varepsilon>0$
\begin{equation}
\label{proto-closed-ldp}
\limsup_{n\to\infty}\ \sup_{x\in C}\ \frac{1}{n}\log \bb{P}\left(X_n=[[nx]]_n\right)
\leq  -\inf_{x\in C} I(x) + \varepsilon.
\end{equation}

We prove \eqref{proto-closed-ldp} by contradiction. if \eqref{proto-closed-ldp} is false, then for each $k$ we can find $n_k$ and $z_k \in C\cap [-1,1]$ for which 
\[
\frac{1}{n_k}\log \bb{P} \left(X_{n_k}=[[n_k z_{k}]]_{n_k}\right)
> -\inf_{x\in C} I(x) + \varepsilon.
\]
By taking a subsequence, we may consider that $z_k \to z \in C$ and therefore if we take $m_k = |[[n_k z]]_{n_k} - [[n_k z_{k}]]_{n_k}|$ and note that $\lim_k m_k/n_k = 0$ by an application of \eqref{e:switch_limit_scgf}, given in Lemma \ref{l:uc} below, it follows that 
\[
\begin{aligned}
-I(z)
&:= -\lim_{k\to\infty}\frac{1}{n_k}\log \bb{P}\left(X_{n_k - m_k}=[[n_k z]]_{n_k}\right)
\\
&\ge
-\lim_{k\to\infty}\frac{1}{n_k}\log \bb{P}\left(X_{n_k}=[[n_k z_{k}]]_{n_k}\right)
\geq 
-\inf_{x\in C} I(x) + \varepsilon,
\end{aligned}
\]
which is in contradiction with 
\[
I(z) \geq \inf_{x\in C} I(x) - \varepsilon.
\]

\begin{lemma}\label{l:uc}
If the environment law $\mu$ is uniformly elliptic, $\lim_k z_k = z$, and $\lim_k n_k =\infty$ then, for $m_k = |[[n_k z]]_{n_k} - [[n_k z_{k}]]_{n_k}|$ 
\begin{equation}\label{e:switch_limit_scgf}
\liminf_k  \frac{\log \bb{P}\prt{X_{n_k - m_k} =\crt{\crt{n_kz}}_{n_k}}
  -\log \bb{P}_0\prt{X_{n_k} =\crt{\crt{n_kz_k}}_{n_k}}}{n_k} \geq 0.
\end{equation}
\end{lemma}
\begin{proof}
First we recall the assumption that of uniform ellipticity consists of the existence of $\mathfrak{c} \in (0,1)$ for which $\mu(\omega  \in ([\mathfrak{c}, 1 - \mathfrak{c}])^{\bb{Z}}) = 1$ and thus an application of the Chapman-Kolmogorov equation implies that for any $x,x' \in \bb{Z}$ and $ m = |x'-x|$
\begin{equation}
\label{UEllip}
\begin{aligned}
  \bb{P}(X_{n } = x')  \leq \bb{P}(X_{n-m} = x) \times \mathfrak{c}^m .    
\end{aligned}
\end{equation}
Therefore with $x_k = [[n_k z]]_{n_k}$ and $x'_k = [[n_k z_k]]_{n_k}$, $m_k = |[[n_k z]]_{n_k} - [[n_k z_k]]_{n_k}|$ it follows that
\[
\log  \bb{P}(X_{n_k -m_k} = x_k) -\log  \bb{P}(X_{n_k} = x'_k) \geq - m_k \log \mathfrak{c}
\]
To conclude the proof it suffices to note that divide both sides by $-\frac{1}{n_k}$ and take the limit in $k$ and to note that $\lim_k \frac{m_k}{n_k} = 0$.

\end{proof}
\section*{Acknowledgement}

We take this opportunity to thank Claudio Landim, whose lessons have generously, and through diverse interactions, travelled “au-delà du Rhin”, helping us develop an enduring interest in the theory of scaling limits of stochastic processes and ultimately enabling us to admire its subtle intrinsic beauty.

\pagebreak

\bibliographystyle{plain}
\bibliography{survey}
\end{document}